\documentclass[11p]{article}

\usepackage{times}
\usepackage{amsthm}
\usepackage{amsmath}
\usepackage{amssymb}
\usepackage{mathrsfs}
\usepackage{pb-diagram}
\usepackage{color}
\usepackage{hyperref}

\def\A{{\mathcal A}}
\def\o{\omega}
\def\O{\Omega}
\def\th{\theta}
\def\Th{\Theta}

\def\T{\mathbb{T}}

\def\E{{\mathscr E}}
\def\F{\mathscr F}
\def\H{\mathcal H}
\def\K{\mathcal K}
\def\R{\mathbb{R}}
\def\X{\mathbf X}
\def\Y{\mathbf Y}

\def\S{\mathcal S}
\def\e{{\sf e}}

\def\m{{\sf m}}

\def\r{{\rm r}}
\def\d{{\rm d}}

\def\({\left(}
\def\[{\left[}
\def\){\right)}
\def\]{\right]}

\def\si{\sigma}
\def\Si{\Sigma}

\def\G{{\sf G}}

\def\p{\parallel}
\def\<{\langle}
\def\>{\rangle}

 \newtheorem{thm}{Theorem}[section]
 \newtheorem{cor}[thm]{Corollary}
 \newtheorem{lem}[thm]{Lemma}
 \newtheorem{prop}[thm]{Proposition}
 \theoremstyle{definition}
 \newtheorem{defn}[thm]{Definition}
 \theoremstyle{remark}
 \newtheorem{rem}[thm]{Remark}
 \newtheorem{ex}[thm]{Example}
 \numberwithin{equation}{section}
\newtheorem{Setting}[thm]{Setting}

\numberwithin{equation}{section}

\begin{document}


\title{$C^*$-Algebraic Spectral  Sets,\\ Twisted Groupoids and Operators}

\author{Marius M\u antoiu}

\author{Marius M\u antoiu \footnote{
\newline
{
}}
}



\maketitle


\begin{abstract}
We treat spectral problems by twisted groupoid methods. To Hausdorff locally compact groupoids endowed
 with a continuous $2$-cocycle one associates the reduced twisted groupoid $C^*$-algebra. Elements
 (or multipliers) of this algebra admit natural Hilbert space representations. We show the relevance 
of the orbit closure structure of the unit space of the groupoid in dealing with spectra, norms, numerical 
ranges and $\epsilon$-pseudospectra of the resulting operators. As an example, we treat a class of 
pseudo-differential operators introduced recently, associated to group actions. We also prove a 
Decomposition Principle for bounded operators connected to groupoids, showing that several relevant 
spectral quantities of these operators coincide with those of certain non-invariant restrictions. 
This is applied to Toeplitz-like operators with variable coefficients and to band dominated operators on discrete metric spaces.
\end{abstract}


\section*{Introduction}\label{introduction}

Many applications of $C^*$-algebras to spectral analysis rely on a very simple principle: a morphism between two $C^*$-algebras is contractive and reduces the spectrum of elements. In addition, if it is injective, it is isometric and preserves spectra. 

\smallskip
While working on this project, we became aware of the fact that such a principle works effectively  for much more than spectra (and essential spectra). We call {\it $C^*$-algebraic  spectral set} a function $\Si$ sending, for any unital $C^*$-algebra $\E$, the elements $E$ of this one to closed (usually compact)  subsets $\Si(E\!\mid\!\E)$ of $\mathbb C$\,, having the above mentioned behavior under $C^*$-morphisms. Definition \ref{ierarhie} states this formally. It turns out that there are many examples, including spectra, various types of $\epsilon$-pseudospectra and various types of numerical ranges. At all levels, it is relevant to factor out closed two-sided ideals. Especially when these ideals are elementary (isomorphic with the ideal of all the compact operators in a Hilbert space), one gets interesting "essential" $C^*$-spectral sets, that may have independent descriptions. The essential spectrum is the reference. 

\smallskip
In this rather abstract setting, we study properties of natural families of $C^*$-algebraic elements and especially  of bounded operators coming out of  Hilbert space representations.
In concrete situations, one has to arrange things to put this principle to work in a relevant way, with a geometrical background. Having a groupoid behind the stage makes things easier. 

\smallskip
The framework of groupoid algebras is a very general yet concrete and versatile setting to study spectral problems. Groupoids have long been used to provide examples of
abstract $C^*$-algebras against which the general theory was tested. They have also been applied to more concrete problems
arising from physics and spectral analysis. But it is not our purpose to sketch here a history of these topics.
 
 \smallskip
In a large part of this paper,  a $2$-cocycle is added to the structure, although a large part of the results are new even without this ingredient. In the classical reference \cite{Re}, twisted $C^*$-algebras (full and reduced) are associated to the couple $(\Xi,\o)$\,, where $\Xi$ is a locally compact groupoid with Haar system and $\o$ is a continuous $2$-cocycle (a cohomological object). Most of the subsequent works (but by no means all) neglected $\o$\,, taking it to be trivial. For the role of {\it twisted} groupoid algebras in classification theory we refer to \cite{Li}. Concerning spectral applications, especially, it seems that the twist did not raise much interest. We cite \cite{BBdN} for an exception; in that article a certain class of groupoids with a cocycle are involved in spectral analysis, especially concerning the continuity of spectra of families of operators. In \cite{MN1} the framework is quite close to the present one; however the results point in other directions. In particular, although the behavior under $C^*$-morphisms plays an important role in \cite{MN1}, it is not enough to get the results. So only essential spectra and essential numerical ranges are treated in \cite{MN1}, with related but different purposes. 

\smallskip
The present article focuses on generalized forms of spectra of various elements and operators constructed from elements or multipliers of the twisted groupoid $C^*$-algebras by localizing in a natural way, using the elements $x\in X$ of its unit space. They are precisely what appears to be interesting in practical situations and enjoy a useful and meaningful covariance property. This one immediately shows that what really matters (within unitary equivalence and equi-spectrality) is not the unit $x$, but its orbit $\mathcal O_x$ with respect to the obvious groupoid equivalence relation. But a much finer result follows from considering {\it orbit closures}.

\smallskip
In section \ref{goam} we review basic facts about groupoids endowed with a $2$-cocycle, following mainly \cite{Re}. This includes defining the regular representations and the reduced twisted groupoid $C^*$-algebra. As stated above, injectivity of morphisms plays an important role. In Proposition \ref{construint} we show that the regular representation corresponding to a unit belonging to a dense orbit is faithful. It is a key property allowing one to transfer spectral information from abstract elements to Hilbert space operators and play an important role in most of the proofs. We only found the untwisted case in the literature \cite{FS,KS,BB}, but existing proofs were easy to adapt to the presence of the cocycle. The groupoid is not supposed to be \'etale or second countable, which would have limited the range of applications. 

\smallskip
Section \ref{meni} is dedicated to our first abstract spectral results. We introduce a partial order relation on the family of orbit closures of the unit space $X$ of the groupoid $\Xi$\,, or more generally on the family of closed invariant subsets. Each such subset defines reductions at the level of the twisted groupoids and at the level of the $C^*$-algebras. We show $\Si$-spectral monotony with respect to these reductions. The Hilbert space operators $\big\{H_x\!\mid\!x\in X\big\}$ obtained by applying regular representations still respect this monotony. In particular, it follows that if two units generate the same orbit closure (which is usually much more general than belonging to the same orbit), then the corresponding operators are $\Si$-isospectral. In particular, they have the same spectra, the same $\epsilon$-pseudospectra, the same numerical ranges and the same polynomial hulls. With the exception of the spectrum these extra spectral quantities, briefly reviewed in Section \ref{menghina}, are rather trivial for normal elements. If the elements are not normal, they are extremely relevant, especially in applications, and subject to a lot of classical or recent research. We only cite \cite{BO,BD,Da,GR,HRS,SW}.

\smallskip
In particular, for minimal orbit closures $\mathcal Q$\,, the results become very simple: they assert that the mentioned $\Si$-spectral sets, computed for the corresponding family $\{H_x\!\mid\!x\in \mathcal Q\}$ of operators, are constant. To our knowledge this is new for the numerical range  and for the polynomial hulls, even if $\o=1$\,. For pseudospectra, our result should be compared with Corollary 4.9 from \cite{BLLS1}. If we fix a Hilbertian setting, our result is more general, and not only because the twist is present. However, in \cite{BLLS1} Banach space operators are treated and groupoid techniques would hardly be adapted to such a situation. 

\smallskip
Constancy of the spectrum of covariant families of operators labeled by the points of a minimal system has a long history. Being closer to our point of view, we only mention papers as \cite{Be,BLLS1,BLLS2,BHZ,BIST,Le,LS1}; they contain other relevant references. Most of the previous results of this type strictly refer to the minimal case, which is only a particular situation here. Considering a fixed orbit closure $\mathcal Q$\,, we put into evidence a decomposition $\mathcal Q=\mathcal Q^{\rm g}\sqcup\mathcal Q^{\rm n}$, where $\mathcal Q^{\rm g}$ is a dense subset of "generic points", on which the spectral quantities are always constant. In other situations, this has been used in \cite{Ma,Ma2}. This is connected with the notion of "pseudo-ergodic elements" used in \cite{BLLS1,BLLS2} for similar purposes.
We also avoided using spectral continuity in proving the result. Our groupoids are quite general (they are not associated to actions of discrete groups) and, in addition, they are twisted. On the other hand, in \cite{BLLS1,BLLS2} there are results for families of operators in Banach spaces that we cannot treat by groupoid methods, a feature we already mentioned above. It would be nice to find unifying methods leading to optimal results and not excluding Banach space operators. 

 \smallskip
 Our interest in the subject originated in the fact that the magnetic pseudo-differential calculus in $\R^n$ (cf.\,\cite{IMP,MP,MPR2} and references therein) may be reformulated in a twisted groupoid framework. In this case one has a transformation groupoid associated to some $\R^n$-action, the cocycle corresponding to a variable magnetic field; it is built from fluxes through triangles of the magnetic field, seen as a differentiable $2$-form on $\R^n$. The resulting situation is, however, very particular. Much more general pseudo-differential formalisms on non-commutative groups have been constructed recently \cite{RT,MR} and a version with cocycles already exists \cite{BM} (still rather undeveloped technically). In Section \ref{goforit} we use it as an exemplification of our abstract results. The untwisted case is treated in \cite{Ma2}.

\smallskip
{\it The second half of the article is dedicated to the Decomposition Principle.} We borrowed the terminology from \cite{LS}.
 Basically, a Decomposition Principle states that the essential spectrum of a suitable Hamiltonian $H$ acting in an $L^2(X)$-space coincides with the essential spectrum of its "restriction" $H_Y$ to  the complement $Y$ of a relatively compact subset $K\subset X$. Besides its intrinsic interest, this is also an essential step towards proving a Persson type formula. In \cite{LS}  such a Decomposition Principle has been obtained for operators associated to regular Dirichlet forms, under an extra (spatial) relative compactness assumption. Analytical and probabilistic methods are used and the range of applications is impressive (Laplace-Beltrami operators on manifolds, Laplace operators on graphs, $\alpha$-stable processes, Schr\"odinger operators). The introduction of \cite{LS} outlines the history of the subject in a much better way than we could do it here.

\smallskip
Our approach is based on groupoid $C^*$-algebras. In this second half of the paper cocycles may also be included but, for simplicity, we only considered the untwisted case; they can be supplied quite easily by the interested reader; see also Remark \ref{aculici}.
For the time being the approach is only confined to bounded operators, but there are classes of examples not covered by \cite{LS} (or by other results in the literature). The range of applications, more modest and almost disjoint from those mentioned above, could be drastically improved by an extra analytical effort, to which we hope to dedicate a future article. Transferring results from bounded to unbounded operators seems possible, by using a natural concept of affiliation. What is really a hard problem to solve is to give strong concrete criteria for the resolvent of  interesting unbounded operators to belong to groupoid $C^*$-algebras. And this should be compatible with the restriction process from $X$ to $Y$. A great achievement  would be the inclusion of twisted pseudo-differential operators on Lie groupoids in the sense of H\"ormander. Without cocycles this is already a well-developed topic (\cite{ALN,DS,NWX,LMN,Va} and many others). The twisted case will need some more technical efforts, especially at the level of affiliation and restrictions.

\smallskip
Anyhow, our results still refer to any $C^*$-algebraic essential spectral set, and not only to the essential spectrum; the main abstract result is Theorem \ref{aguacate}. A Fredholmness result is also included. The proofs need to perform restrictions to {\it non-invariant} sets of units. At the groupoid $C^*$-algebra level, this consists of involutive linear contractions {\it that are not $C^*$-morphisms}, and this is the main topic of section \ref{grafitoni}. In section \ref{grafomoni} we get an intrinsic Decomposition Principle, that refers to essential $C^*$-algebraic spectral sets of elements of groupoid algebras modulo suitable ideals. In section \ref{gramofoni} we assume that our groupoid is {\it standard}; the main new feature is the existence of an open dense orbit in the unit space, with trivial isotropy group. This ensures the existence of a canonical faithful Hilbert space representation, in which the relevant ideal in the groupoid algebra becomes  the ideal of all the compact operators.
In this representation the Decomposition Principle becomes concrete and, when $\Si={\sf sp}$ is the usual spectrum, one recovers the equality of essential spectra of the initial and the restricted operators. But other cases are also interesting.

\smallskip
Section \ref{gerfomoni} presents a class of examples, containing certain Toeplitz operators (even with "variable coefficients") on locally compact non-compact groups as particular cases. Their essential spectra coincide with the essential spectra of the corresponding convolution operators, that are simpler. For Abelian groups this becomes quite explicit, since the spectrum of a convolution operator in such a case is purely essential and, by Fourier unitary equivalence, it is just the spectrum of an operator of multiplication by a function defined on the Pontryagin dual of the group.

\smallskip
In the final section we show a decomposition principle in the setting of Roe's uniform $C^*$-algebra associated to a uniformly discrete metric space with bounded geometry satisfying Yu's Property A. The operators involved are band dominated operators and the restriction procedure consists in removing a finite number of raws and columns in the matrix-realization. The works applying the limit operator techniques in such a setting are related to a certain extent to our results. We cite recent articles on this topic \cite{AZ,HS,SWi}; see also references therein. The end of the section is dedicated to an extension of the Decomposition Principle to larger classes of operators, permitted by the discreteness of the metric space; we develop a suggestion of a referee, to whom we are grateful.

\smallskip
We end this Introduction with a couple of technical remarks.

\smallskip
Employing $C^*$-algebraic techniques allowed us to treat naturally elements and operators that are not normal. This is mainly due to the fact that using the functional calculus has been avoided.  This is important, since some of our $C^*$-algebraic spectral sets are particularly interesting for non-normal operators. In Proposition \ref{strofok} only, normality was required for one of the implications, since one needs equality between the norm and the spectral radius. On the other hand, it is difficult to imagine how to treat operators in Banach spaces by a similar approach. 

\smallskip
Most of the time, our groupoids are supposed to be Hausdorff, locally compact, $\si$-compact and to possess a Haar system. Some constructions or results would not need all these conditions, but this would require extra effort that is not justified having in view the examples we treat. In particular, the important Proposition \ref{construint} seems to rely on $\si$-compactness, although we are not aware of counterexamples; see also \cite{KS}. Starting from Section \ref{grafitoni} we also require the groupoid to be (topologically) amenable. For this concept we refer to \cite{ADR,Re1}. Amenability is a standard way to insure that the full and the reduced groupoid algebras coincide; this is shown in \cite{Re1} for $\si$-compact locally compact groupoids admitting a Haar system. Gaining in generality, this property (called the weak containment property) could be required in an implicit way. Unfortunately, it is not easy to see if it is preserved under various constructions, which makes necessary imposing it repeatedly. Starting from Section \ref{gramofoni} it is also imposed that the unit space $\,X$ has a dense open orbit $X_0$ and the isotropy of this orbit is trivial.  This implies the existence of a certain important faithful representation, in which the reduction of the groupoid to $X_0$ can be identified with the ideal of all compact operators.

\smallskip
We did not ask our groupoids to be second countable. The main reason is Section \ref{germanofoni}, where the Stone-\u Cech compactification of a (suitable) discrete metric space appears; it is $\si$-countable without being second countable.

\section{Spectral sets in a $C^*$-algebraic setup}\label{menghina}

\begin{defn}\label{ierarhie}
To give a {\it $C^*$-algebraic spectral set} means to indicate, for every unital complex $C^*$-algebra $\mathscr E$, a function 
$$
\Si(\cdot\!\mid\!\E):\mathscr E\to{\sf Cl}(\mathbb C):=\{\Lambda\subset\mathbb C\!\mid\!\Lambda\ {\rm closed}\}
$$
such that if $\pi:\mathscr E\to\mathscr F$ is a unital $C^*$-morphism then 
\begin{equation}\label{pulgar}
\Si\big(\pi(E)\!\mid \!\mathscr F\big)\subset\Si(E\!\mid\!\E)\,,\quad\forall\,E\in\mathscr E
\end{equation}
and such that in \eqref{pulgar} one has equality if $\pi$ is injective.
\end{defn}

It is convenient to call $\Si(E\!\mid\!\E)$ {\it the $\Si$-spectrum of $E\in\mathscr E$}. In particular, if $\mathscr E$ is a $C^*$-subalgebra of $\mathscr F$, for every $E\in\mathscr E$ the equality $\Si(E\!\mid\!\mathscr F)=\Si(E\!\mid\!\mathscr E)$ will hold. The importance of the faithful $^*$-representations is clear in this context.

\begin{ex}\label{celebazice}
{\rm The basic example of a $C^*$-algebraic spectral set is {\it the spectrum} $\Si(\cdot\!\mid\!\E):={\sf sp}(\cdot\mid\E)$\,.} The fact that it is indeed a $C^*$-algebraic spectral set follows immediately from the definition of the complement of the spectrum in terms of invertibility.
\end{ex}

\begin{ex}\label{celebazuce}
{\rm If $\mathscr E$ is a unital $C^*$-algebra , then $\mathcal S(\mathscr E)$ denotes its {\em state space}, that is, the set of positive linear forms
$\phi:\mathscr E\to\mathbb C$ such that $\phi(1_\mathscr E)=1$\,. The (algebraic) {\em numerical range} \cite{BD,BD2,HRS,SW} of $E\in\mathscr E$ is
\begin{equation*}
{\sf nr}(E\!\mid\!\E):=\{\phi(E)\!\mid\!\phi\in\S(\mathscr E)\}\,.
\end{equation*}
It is a compact, convex subset of $\mathbb C$ containing the convex hull of the spectrum:
\begin{equation*}
{\rm co}[{\sf sp}(E\!\mid\!\E)] \subset {\sf nr}(E\!\mid\!\E) \subset \big\{\lambda \in \mathbb C\mid |\lambda| \le\, \|E\|_\mathscr E\!\big\}\,.
\end{equation*} 
For normal elements $E\in\mathscr E$ the first inclusion is an equality, but in general the inclusions are strict. If $H\in\mathscr E=\mathbb B(\H)$ (the bounded linear operators in some Hilbert space $\H$)\,, then ${\sf nr}\big(H\!\mid\!\mathbb B(\H)\big)$ is the closure of {\it the operator (spatial) numerical range} \cite{GR}
\begin{equation*}
{\sf nr}_0(H):=\big\{\<Hu,u\>_\H\,\big\vert\,u\in\H\,,\|u\|_\H\,=\!1\big\}\,.
\end{equation*}
}
\end{ex}

\begin{lem}\label{sontorc}
The numerical range is a $C^*$-algebraic spectral set.
\end{lem}

\begin{proof}
We check that {\it the numerical range shrinks under unital $C^*$-morphisms and is left unchanged if the $C^*$-morphism is injective}. 
If $\mu:\mathscr E\to\mathscr F$ is a unital $C^*$-morphism, the transpose 
$$
\mu^*:\mathscr F^*\to\mathscr E^*,\quad\mu^*(\varphi):=\varphi\circ\mu
$$
restricts to a map $\mu^*\!:\S(\mathscr F)\to\S(\mathscr E)$\,. Then, for each $E\in\mathscr E$\,:
$$
\begin{aligned}
{\sf nr}\big(\mu(E)\!\mid\!\mathscr F\big)&=\{\varphi[\mu(E)]\!\mid\!\varphi\in\S(\mathscr F)\}=\{[\mu^*(\varphi)](E)]\!\mid\!\varphi\in\S(\mathscr F)\}\\
&\subset\{\phi(E)\!\mid\phi\in\S(\mathscr E)\}={\sf nr}(E\!\mid\!\mathscr E)\,.
\end{aligned}
$$
If $\mu$ is injective, then $\mu^*$ is surjective (by the Hahn-Banach Theorem) and we get equality.
\end{proof}

There are other proofs, based on various characterizations of the numerical range.

\smallskip
Let $\mathscr E$ be a unital $C^*$-algebra and $E_0,E_1,\dots,E_n\in\mathscr E$. For each $\lambda\in\mathbb C$ one defines \cite[Sect.\,3.3]{HRS} {\it the operator polynomial} 
\begin{equation}\label{opol}
\mathbf E(\lambda)\!:=E_0+\lambda E_1+\dots+\lambda^n E_n\in\mathscr E.
\end{equation}

\begin{defn}\label{lumeanoua}
\begin{enumerate}
\item[(i)]
{\it The $0$-spectrum of} $\,\mathbf E$ is defined as $\bf{sp}_0(\mathbf E\!\mid\!\E)\!:=\{\lambda\in\mathbb C\!\mid\!\mathbf E(\lambda)\ {\rm is\ not\ invertible}\}$\,.
\item[(ii)]
For every $\epsilon>0$ {\it the $\epsilon$-pseudospectrum of} $\,\mathbf E$ is defined as 
\begin{equation}\label{pseudospe}
{\bf psp}_\epsilon(\mathbf E\!\mid\E):=\big\{\lambda\in\mathbb C\mid\,\p\!\mathbf E(\lambda)^{-1}\!\p_\mathscr E\,\ge 1/\epsilon\big\}\,.
\end{equation} 
\end{enumerate}
\end{defn}

By convention, one sets $\p\!\mathbf E(\lambda)^{-1}\!\p_\E\,=\infty$ if $\,\mathbf E(\lambda)$ is not invertible, so ${\bf sp}_0(\mathbf E\!\mid\!\E)\subset\bigcap_{\epsilon>0}{\bf psp}_\epsilon(\mathbf E\!\mid\!\E)$\,, very often strictly.

\smallskip
We speak of {\it the standard (or the linear) case} when $n=1$\,, $E_0\equiv E\in\mathscr E$ and $E_1\!:=-1_\mathscr E$\,. Then $\mathbf E(\lambda)=E-\lambda 1_\mathscr E$ and we recover the usual notions of {\it spectrum} and {\it $\epsilon$-pseudospectrum} of the element $E\in\E$. If $E$ is normal, the $\epsilon$-pseudospectrum is just an $\epsilon$-neighborhood of the spectrum and the situation is simple; we are mainly interested in the non-normal case. We refer to \cite[Sect.\,9]{Da} for a detailed study, many examples and applications of pseudospectra (in a Hilbert space setup), as well as for generalizations to which our approach also works. It is argued that for non-normal operators the concept is more useful and easier to control than the usual spectrum.

\begin{lem}\label{sontork}
The $\epsilon$-pseudospectra are $C^*$-algebraic spectral sets.
\end{lem}

\begin{proof}
The assertion about the behavior of  the $\epsilon$-pseudospectrum under unital $C^*$-morphisms follows from the definition \ref{pseudospe}, from the fact that $C^*$-algebras are stable under inversion, $C^*$-morphisms decrease norms and are isometric if they are injective. 
\end{proof}

\smallskip
Let us mention briefly some other concepts. Much more about these notions and their applications may be found in \cite[Sect.\,9.4]{Da} and in the references therein. Let $H$ be a bounded operator in the Hilbert space $\H$ and $p:\mathbb C\to \mathbb C$ a polynomial. One sets
\begin{equation}\label{celeuna}
{\sf hull}(p,H)\!:=\{\lambda\in\mathbb C\mid|p(\lambda)|\le\,\p\!p(H)\!\p\}\,,
\end{equation}
\begin{equation}\label{celedoua}
{\sf num}(p,H):=\{\lambda\in\mathbb C\mid p(\lambda)\in{\sf nr}[p(H)]\}\,.
\end{equation}
It can be shown that
$$
{\sf sp}(H)\subset{\sf num}(p,H)\subset{\sf hull}(p,H)\subset{\sf nr}(H)\,,
$$
and all the inclusions may be strict. Actually, the intersection of ${\sf num}(p,H)$ and ${\sf hull}(p,H)$ (respectively) over all polynomials are equal and coincide with the polynomial convex hull of ${\sf sp}(H)$ \cite[Th.\,9.4.6]{Da}. It is easy to see that the two notions appearing in \eqref{celeuna} and \eqref{celedoua} also make sense for elements in $C^*$-algebras and define $C^*$-algebraic spectral sets.

\smallskip
Essential $C^*$-algebraic spectral sets (computed in quotients) will be mentioned when needed.

\section{Twisted groupoid $C^*$-algebras}\label{goam}

As general notations, we deal with groupoids $\Xi$ over a unit space $\Xi^{(0)}\!\equiv X$. We recall that a groupoid can be seen as a small category in which all the morphisms are invertible.  For the general theory of locally compact groupoids and their $C^*$-algebras we refer to \cite{Re,Wi}\,. We treat objects of the category as morphisms in the canonical way.
The source and range maps are denoted by ${\rm d,r}:\Xi\to \Xi^{(0)}$ and the family of composable pairs by $\,\Xi^{(2)}\!\subset\Xi\times\Xi$\,. For $A,B\subset X$ one uses the standard notations 
$$
\Xi_A:={\rm d}^{-1}(A)\,,\quad\Xi^B:={\rm r}^{-1}(B)\,,\quad\Xi_A^B:=\Xi_A\cap\Xi^B.
$$
Particular cases are the $\d$-fibre $\Xi_x\equiv\Xi_{\{x\}}$\,, the  ${\rm r}$-fibre $\Xi^x\equiv\Xi^{\{x\}}$ and the isotropy group  $\Xi_x^x\equiv\Xi_{\{x\}}^{\{x\}}$ of a unit $x\in X$. 

\smallskip
As it is trivial to check, in every groupoid an equivalence relation on $X$ is defined by $x\approx y$ if $x=\r(\xi)$ and $y={\rm d}(\xi)$ for some $\xi\in\Xi$\,. {\it The orbit of a point} $x$ will be denoted by $\mathcal O_x={\rm r}(\Xi_x)$\,. If there is just one orbit, we say that the groupoid is {\it transitive}. A subset $Y\subset X$ is called {\it invariant} if  for every $x\in Y$ one has $\mathcal O_x\subset Y$.

\smallskip
We are going to suppose that $\Xi$ is a Hausdorff locally compact $\si$-compact groupoid \cite{Re,Wi}. In particular, $\Xi$ and $X$ are Hausdorff and locally compact and all the structure maps are continuous. The unit space $X$, as well as all the fibers $\Xi^x,\Xi_x$ are closed in $\Xi$\,.

\smallskip
We recall the concept of {\it right Haar system}. It is a family of positive regular Borel measures $\big\{\lambda_x\,\big\vert\, x\in X\big\}$ labelled by the unit space of the groupoid such that
\begin{itemize}
\item
$\lambda_x$ is fully supported on $\Xi_x$\,,
\item
for every $f:\Xi\to\mathbb C$ continuous and compactly supported, the map $X\ni x\to\int_{\Xi_x}\!fd\lambda_x$ is continuous,
\item
for any $\xi\in\Xi$\,, the homeomorphism $\Xi_{\r(\xi)}\ni\eta\to\eta\xi\in\Xi_{\d(\xi)}$ transports $\lambda_{\r(x)}$ to $\lambda_{\d(x)}$\,.
\end{itemize}
Neither the existence nor the uniqueness of such an object is guaranteed in general. But when Haar systems do exist, it is known \cite{Re,Wi} that $\d,\r$ are open maps.

\smallskip
{\it All over the remaining part of the paper $\Xi$ will be a Hausdorff, locally compact $\si$-compact groupoid with fixed right Haar system $\lambda=\{\lambda_x\!\mid\!x\in X\}$\,.}
 Composing with the inversion $\Xi\ni \xi\to\iota(\xi)\equiv \xi^{-1}\in\Xi$\,, we get the left Haar system $\{\lambda^x\!\mid\!x\in X\}$ (obvious definition, analog with the one of a right Haar system, mutatis mutandis).

\begin{defn}\label{asteluza}
{\it A $2$-cocycle} \cite{Re} is a continuous function $\,\o:\Xi^{(2)}\to\T:=\{z\in\mathbb C\mid |z|=1\}$ satisfying
\begin{equation}\label{doicocic}
\o(\xi,\eta)\o(\xi\eta,\zeta)=\o(\eta,\zeta)\o(\xi,\eta\zeta)\,,\quad\ \forall\,(\xi,\eta)\in\Xi^{(2)},\ (\eta,\zeta)\in\Xi^{(2)},
\end{equation}
\begin{equation}\label{normaliz}
\o(x,\eta)=1=\o(\xi,x)\,,\quad\ \forall\,\xi,\eta\in\Xi\,,\,x\in X,\;\r(\eta)=x=\d(\xi)\,.
\end{equation}
Particular cases are the {\it $2$-coboundaries}, those of the form
\begin{equation*}\label{coboundary}
[\delta^1(\si)](\xi,\eta):=\si(\xi)\si(\eta)\si(\xi\eta)^{-1}
\end{equation*}
for some continuous map $\si:\Xi\to\T$\,. We write $\o\in Z^2(\Xi;\T)$ and $\delta^1(\si)\in B^2(\Xi;\T)$\,. The quotient $H^2(\Xi;\T):=Z^2(\Xi;\T)/B^2(\Xi;\T)$ is called {\it the second group of cohomology of the groupoid $\Xi$.}
\end{defn}

If $(\Xi,\lambda,\o)$ is given, $C_c(\Xi)\equiv C_c(\Xi,\o)$ becomes a $^*$-algebra with the composition law
\begin{equation}\label{mutilaw}
\begin{aligned}
(f\star_\o\!g)(\xi)&:=\int_{\Xi}\!f(\eta^{-1})g(\eta\xi)\,\o\big(\eta^{-1}\!,\eta\xi\big)\,d\lambda_{\r(\xi)}(\eta)\\
&=\int_{\Xi}\!f(\xi\eta^{-1})g(\eta)\,\o\big(\xi\eta^{-1}\!,\eta\big)\,d\lambda_{\d(\xi)}(\eta)
\end{aligned}
\end{equation}
(the equivalence of the two expressions in \eqref{mutilaw} is easy to establish) and the involution
\begin{equation}\label{ridic}
f^{\star_\o}(\xi):=\overline{\o(\xi,\xi^{-1})}\,\overline{f(\xi^{-1})}\,.
\end{equation}
Then, by the usual procedure \cite{Re}, one gets {\it the (full) twisted groupoid $C^*$-algebra} ${\sf C}^*(\Xi,\o)$\,. We do not describe its construction, because we are going to work with the reduced version, to be introduced below.

\begin{ex}\label{cuartet}
{\rm By \cite[Remark 2.2]{BaH}, on principal transitive groupoids (pair groupoids) all $2$-cocycles are trivial (coboundaries). The (twisted) groupoid algebra is elementary, i.e. isomorphic to the $C^*$-algebra of all the compact operators on a Hilbert space.}
\end{ex}

Recall that for every $x\in X$ there is a representation $\Pi_x$ defined by
$$
\Pi_x(f)u:=f\star_\o\!u\,,\quad\forall\,f\in C_{\rm c}(\Xi)\,,\ u\in L^2(\Xi_x;\lambda_x)=:\H_x\,.
$$
These {\it regular representations} are associated to (induced from) Dirac measures $\nu=\delta_x$ on the unit space. Among others, the regular representations serve to define the reduced norm
\begin{equation*}\label{redunorm}
\p\!\cdot\!\p:C_{\rm c}(\Xi)\to\R_+\,,\quad\p\!f\!\p\,:=\sup_{x\in X}\!\p\!\Pi_x(f)\!\p_{\mathbb B(\H_x)}.
\end{equation*}
The completion of $\big(C_{\rm c}(\Xi),\p\!\cdot\!\p\!\big)$ is denoted by ${\sf C}^*_{\rm r}(\Xi,\o)$\,, it is called {\it the reduced twisted $C^*$-algebra of the twisted groupoid $(\Xi,\lambda,\o)$} and generally it is a quotient of ${\sf C}^*(\Xi,\o)$\,. The isomorphy class of both the full and the reduced twisted groupoid $
C^*$-algebras only depend on the cohomology class of the cocycle (use $f\to\si f$ if $\,\o'=\delta^1(\si)\o$)\,.

\begin{rem}\label{jiek}
{\rm It follows easily that for every $\xi\in\Xi$ one has the unitary equivalence $\,\Pi_{\r(\xi)}\approx\Pi_{\d(\xi)}$\,, so the regular representations along an orbit are all unitarily equivalent: 
$x\approx y\,\Rightarrow\,\Pi_x\approx\Pi_y\,.$}
\end{rem}

The next result is well-known for {\it untwisted}  groupoid $C^*$-algebras (see \cite{FS,KS,BB} for instance). We are going to need it in our general case. See also \cite[Corollary 2.16]{MN1}.

\begin{prop}\label{construint}
Let $\o$ be a $2$-cocycle on the Hausdorff locally compact $\si$-compact groupoid $\Xi$ with Haar system. Assume that there is a dense orbit $M\subset X$. For every $x\in M$ the representation $\Pi_x$ is faithful. If $\,{\sf C}_{\rm r}^*(\Xi,\o)$ is not unital, the extension $\Pi_x^{\bf M}:{\sf C}_{\rm r}^*(\Xi,\o)^{\bf M}\to\mathbb B\big[L^2(\Xi_x;\lambda_x)\big]$ to the multiplier $C^*$-algebra is also faithful. 
\end{prop}

Note that ${\sf C}_{\rm r}^*(\Xi,\o)$ is unital if the groupoid is \'etale, with compact unit space.

\begin{proof}
The proof is based on adaptations to the twisted case of arguments from \cite{KS,BB}. 

\smallskip
One defines the $C_0(X)$-valued scalar product $\<\cdot\!\mid\!\cdot\>:C_{\rm c}(\Xi,\o)\times C_{\rm c}(\Xi,\o)\to C_0(X)$ by
\begin{equation*}\label{c0prod}
\<f\!\mid\!g\>(x)\!:=\big(f^{\star_\o}\!\star_\o\!g\big)|_X(x)=\int_{\Xi_x}\!\overline{f(\eta)}g(\eta)d\lambda_x(\eta)
\end{equation*}
(it is easy to check that the cocycle disappears from the formula). This and the right action 
\begin{equation*}\label{zocotealla}
C_{\rm c}(\Xi,\o)\times C(X)\ni(f,\psi)\to f\cdot\psi:=(\psi\circ \d)\,f\in C_{\rm c}(\Xi,\o)
\end{equation*}
 (also $\o$-independent) make $C_{\rm c}(\Xi,\o)$ a pre-Hilbert $C_0(X)$-module. Its completion is denoted, conventionally, by $L^2(\Xi,\lambda)$\,. The norm on $C_{\rm c}(\Xi,\o)$ is given by
\begin{equation*}\label{normuza}
\p\!g\!\p^2_{L^2}\,:=\,\p\!\<g\!\mid\!g\>\!\p_\infty=\sup_{x\in X}\int_{\Xi_x}\!\!|g(\eta)|^2d\lambda_x(\eta)=\sup_{x\in X}\!\p\!g|_{\Xi_x}\!\!\p_{\H_x}^2.
\end{equation*}
Using this, for $f,g\in C_{\rm c}(\Xi,\o)$ we get
$$
\begin{aligned}
\p\!f\!\star_\o\!g\!\p_{L^2}\,&=\,\sup_{x\in X}\!\p\!(f\!\star_\o\!g)|_{\Xi_x}\!\!\p_{\H_x}\,=\,\sup_{x\in X}\!\p\![\Pi_x(f)](g|_{\Xi_x})\!\p_{\H_x}\\
&\le\,\sup_{x\in X}\!\p\!\Pi_x(f)\!\p_{\mathbb B(\H_x)}\sup_{x\in X}\!\p\!g|_{\Xi_x}\!\!\p_{\H_x}\,,
\end{aligned}
$$
ending up in $\p\!f\!\star_\o\!g\!\p_{L^2}\,\le\,\p\!f\!\p\,\p\!g\!\p_{L^2}.$

\smallskip
A Hilbert $C_0(X)$-module is surely the space of continuous sections over a continuous field $\{\K_x\!\mid\!x\in X\}$ of Hilbert spaces, while the adjointable elements $T\in\mathbb B\big[L^2(\Xi,\lambda)\big]$ may be written as $^*$-strongly continuous fields of operators $\{T_x\in\mathbb B(\mathcal K_x)\!\mid\!x\in X\}$\,, the norm of $T$ coinciding with $\,\sup_{x}\!\p\!T_x\!\p_{\mathbb B(\K_x)}$\,. One shows exactly as in \cite[Th.\,2.3]{KS} that there is a representation $\Pi:{\sf C}^*(\Xi,\o)\to\mathbb B\big[L^2(\Xi,\lambda)\big]$ such that ${\sf C}^*_{\rm r}(\Xi,\o)$ identifies to $\Pi\big[{\sf C}^*(\Xi,\o)\big]\subset\mathbb B\big[L^2(\Xi,\lambda)\big]$\,. The Hilbert spaces $\K_x$ and $\H_x$ are unitarily equivalent; under this equivalence, $\Pi(f)_x$ becomes $\Pi_x(f)$ for every $x\in X$ and $f\in{\sf C}^*(\Xi,\o)$\,.

\smallskip
Using the abstract Lemma 2.2 from \cite{KS}, one shows immediately that the reduced norm can be recuperated from any dense subset $D\subset X$ of units:
\begin{equation}\label{redundnorm}
\p\!f\!\p\,=\sup_{x\in D}\!\p\!\Pi_x(f)\!\p_{\mathbb B(\H_x)}.
\end{equation}
Without cocycle, this is \cite[Corollary 2.4(a)]{KS}. Now one takes for $D$ the dense orbit $M$ and takes into consideration our Remark \ref{jiek}, asserting that along the orbit the representations $\Pi_x$ are mutually unitarily equivalent. If $\,\Pi_{x_0}(f)=0$ for some $x_0\in M$, then \eqref{redundnorm}, with $D=M$, shows that $f$ is zero.

\smallskip
In the non-unital case, faithfulness is preserved for any unitization (extension of the representation) $\Pi_x^{\bf U}$, since ${\sf C}^*(\Xi,\o)$ is an essential ideal of ${\sf C}^*(\Xi,\o)^{\bf U}$. We recall that the multiplier algebra is the maximal unitization.
\end{proof}

If $A\subset X$ is invariant, $\Xi_A^A=\Xi_A=\Xi^A$ is a subgroupoid, called {\it the reduction of $\,\Xi$ to $A$}\,. If $A$ is also locally closed (the intersection between an open and a closed set), then $A$ is locally compact and $\Xi_A$  will also be a Hausdorff locally compact groupoid, on which one automatically considers the restriction of the Haar system.

\begin{rem}\label{coerenta}
{\rm If $A$ is closed and invariant, one has an epimorphism of groupoid $C^*$-algebras
\begin{equation*}\label{ses}
\rho_A:{\sf C}_{\rm r}^*(\Xi,\o)\to{\sf C}_{\rm r}^*(\Xi_A,\o_A)\,,
\end{equation*}
where $\o_A$ is the restriction of $\o$ to $\Xi_A^{(2)}$. On continuous compactly supported functions the epimorphism $\rho_A$ acts as a restriction. In such a case, let $x\in A$\,; one has the regular representations $\Pi_x:{\sf C}_{\rm r}^*(\Xi,\o)\to\mathbb B\big[L^2(\Xi_x;\lambda_x)\big]$ and $\Pi_{A,x}:{\sf C}_{\rm r}^*\big(\Xi_A,\o_A\big)\to\mathbb B\big[L^2(\Xi_x;\lambda_x)\big]$ (the fibers $\Xi_x$ and $\Xi_{A,x}$ coincide). It is easy to check, and will be used below, that $\Pi_{x}=\Pi_{A,x}\circ\rho_A$\,.}
\end{rem}

\section{$C^*$-spectral quantities and orbit closures in groupoids}\label{meni}

All over this section {\it a $C^*$-algebraic spectral set $\Si$ is fixed}. The notations below will make the labeling  of $\Si(E\!\mid\!\E)$ by $\E$ unnecessary; it will always be clear to which (twisted groupoid) $C^*$-algebra the element belongs. All the examples from Section \ref{menghina} are available, and probably many more.

\smallskip
We are going to assume that {\it the Hausdorff locally compact $\si$-compact groupoid $\Xi$\,, with unit space $X$, has a Haar system and that $\o$ is a continuous $2$-cocycle on $\,\Xi$}\,. The corresponding groupoid algebra is generally non-unital. Units are needed for spectral theory, and on this occasion we would like to have large unital extensions, to accommodate  as many elements as possible, so we make use of multiplier $C^*$-algebras. Recall that a surjective $^*$-morphism $\Phi:\mathscr C\to\mathscr D$ extends uniquely to a $^*$-morphism $\Phi^{\bf M}\!:\mathscr C^{\bf M}\to\mathscr D^{\bf M}$ between the multiplier algebras; if $\mathscr C$ and $\mathscr D$ are $\si$-unital (separable, in particular), then $\Phi^{\bf M}$ is also surjective. For each locally closed invariant subset $B$ of $X$ one often sets $\mathscr C_B:={\sf C}_{\rm r}^*(\Xi_B,\o_B)$ and $\mathscr C_B^{\bf M}:={\sf C}_{\rm r}^*(\Xi_B,\o_B)^{\bf M}$\,; the index $B=X$ will be skipped. It is also convenient to abbreviate $\H_y:=L^2(\Xi_y;\lambda_y)$ for every $y\in X$. 

\smallskip
Let us denote by $\mathbf{inv}(X)$ the family of all closed, $\Xi$-invariant subsets of $X$. 
For $A,B\in\mathbf{inv}(X)$ with $B\supset A$ one has the restriction morphism 
$$
\rho_{BA}:C_{\rm c}(\Xi_{B})\to C_{\rm c}(\Xi_{A})\,,\quad\rho_{BA}(f):=f|_{\Xi_A}\,,
$$
that extends to an epimorphism $\rho_{BA}:{\sf C}_{\rm r}^*(\Xi_{B},\o_B)\to {\sf C}_{\rm r}^*(\Xi_{A},\o_A)$\,. 
Clearly $\rho_{A}=\rho_{XA}$\,. For an element $F$ belonging to $\mathscr C^{\bf M}$, for every $B\in \mathbf{inv}(X)$ and for every $x\in X$, one sets 
\begin{equation*}\label{tupid}
F_B:=\rho_B^{\bf M}(F)\in\mathscr C_B^{\bf M}
\end{equation*} 
and
\begin{equation}\label{stupid}
H_x:=\Pi_x^{\bf M}(F)\in\mathbb B(\H_x)\,.
\end{equation}
 We are interested in the $C^*$-algebraic sets of type $\Si$ of the elements $F_B$ and of the operators $H_x$\,. 

\smallskip
We start with  a simple result about these restrictions.

\begin{prop}\label{grima}
Let $\,F\in\mathscr C^{\bf M}$\,. For every $A,B\in\mathbf{inv}(X)$ such that $B\supset A$ one has 
$$
\Si(F_B)\supset\Si(F_A)\,.
$$
\end{prop}

\begin{proof}
Use the relation $\rho_{BA}\circ \rho_{B}=\rho_{A}$\,, involving epimorphisms, thus leading to $\rho_{BA}^{\bf M}\circ \rho_{B}^{\bf M}=\rho_{A}^{\bf M}$. Then
$$
\Si(F_A)=\Si\big[\rho_A^{\bf M}(F)\big]=\Si\big[\rho_{BA}^{\bf M}\big(\rho_B^{\bf M}(F)\big)\big]\subset\Si\big[\rho_B^{\bf M}(F)\big]=\Si(F_B)\,.
$$
\end{proof}

A subfamily of $\mathbf{inv}(X)$ is particularly interesting, so we are going to need some terminology. 

\begin{defn}\label{cimilituri}
\begin{enumerate}
\item[(i)]
{\it The orbit closure  generated by $x\in X$} is 
$$
\mathcal Q_x\!:=\overline{\mathcal O_x}=\overline{\{z\in X\mid x\approx z\}}\,.
$$ 
Even without specifying $x$\,, {\it an orbit closure} is simply the closure of an orbit.
Then $\,\mathbf{qo}(X)\subset\mathbf{inv}(X)$ denotes the family of all orbit closures. 
\item[(ii)]
We define {\it a quasi-order} on $X$ by $x\prec y\ \Leftrightarrow\ \mathcal Q_x\subset\mathcal Q_y$\,, with the associated {\it orbit closure equivalence relation} (generally weaker than $\,\approx$)
\begin{equation*}\label{qoequiv}
x\sim y\ \Leftrightarrow\ x\prec y\,,\,y\prec x\ \Leftrightarrow\  \mathcal Q_x=\mathcal Q_y\,.
\end{equation*}
\end{enumerate}
\end{defn}

Clearly, the two equivalence relations coincide if and only if different orbits have different closures. If the groupoid is minimal (see below), then all the units are $\sim$ - equivalent, but there could be many different (dense) orbits even in the particular case of group actions.

\smallskip
More generally, if $A\in\mathbf{inv}(X)$\,, the families $\,\mathbf{qo}(A)\subset\mathbf{inv}(A)$ are defined similarly, in terms of the reduced groupoid $\Xi_A$\,; clearly $\mathbf{qo}(A)$ can be seen as the family of all the $\Xi$-orbit closures that are contained in $A$\,.  In particular, if $A=\mathcal Q\equiv\mathcal Q_y$ is an orbit closure, 
$$
\mathbf{qo}(\mathcal Q)=\{\mathcal Q_x\mid x\prec y\}=\{\mathcal Q'\in\mathbf{qo}(X)\mid\mathcal Q'\subset\mathcal Q\}
$$ 
may be very rich in many examples. Subsequently we will also need the notation $\mathbf{qo}_0(\mathcal Q):=\mathbf{qo}(\mathcal Q)\!\setminus\!\{\mathcal Q\}$ for the family of orbit closures strictly contained in $\mathcal Q$\,.

\smallskip
Obviously, by the unitary equivalence of Remark \ref{jiek}, still true for the extensions to the unitizations, one has $\Si(H_x)=\Si(H_y)\ {\rm if}\  x\approx y$\,. The next result, relying on the spectral monotony of Proposition \ref{grima} and on the injectivity result of Proposition \ref{construint}, gives us the finer information that the $C^*$-algebraic spectral quantities actually depends only on the generated orbit closure. 

\begin{thm}\label{operatorizat}
Let $x,y\in X$ and $F\in\mathscr C^{\bf M}$. The operators $H_x$ are as in \eqref{stupid}.
\begin{enumerate}
\item[(i)]
If $x\prec y$ (i.e. if $\mathcal Q_x\subset\mathcal Q_y$), then $\Si\big(H_x\big)\subset\Si\big(H_y\big)$\,. 
\item[(ii)]
In particular, if $x$ and $y$ generate the same orbit closure, written above as $x\sim y$\,, then 
$$
\Si\big(H_x\big)=\Si\big(H_y\big)\,.
$$ 
\end{enumerate}
\end{thm}

\begin{proof}
Clearly, one only needs to show the first statement (i). 

\smallskip
Recall that, $\mathcal Q_x$ being invariant, the ${\rm d}$-fiber $\big(\Xi_{\mathcal Q_x}\big)_x$ of the reduced groupoid $\Xi_{\mathcal Q_x}$ coincides with $\Xi_x$\,. In terms of the extended induced representation $\Pi^{\bf M}_{\mathcal Q_x,x}:\mathscr C_{\mathcal Q_x}^{\bf M}\to\mathbb B(\H_x)$\,, one has by Remark \ref{coerenta}
$$
H_x=\Pi_x^{\bf M}(F)=\Pi_{\mathcal Q_x,x}^{\bf M}[\rho_{\mathcal Q_x}^{\bf M}(F)]=\Pi_{\mathcal Q_x,x}^{\bf M}[F_{\mathcal Q_x}]\,.
$$ 
In addition, the reduced groupoid $\Xi_{\mathcal Q_x}$ has a dense orbit $\mathcal O_x$\,. By Proposition \ref{construint} applied to the groupoid $\Xi_{\mathcal Q_x}$, the representation $\Pi_{\mathcal Q_x,x}$ is injective, thus $\Pi_{\mathcal Q_x,x}^{\bf M}$ is also injective. Similar statements hold for the point $y$\,.  So, by Proposition \ref{grima}, one has
$$
\Si(H_x)=\Si(F_{\mathcal Q_x})\subset\Si(F_{\mathcal Q_y})=\Si(H_y)\,.
$$
\end{proof}

A closer look at points on a given orbit closure gives us more refined spectral information.

\begin{defn}\label{generic}
Let $x\in\mathcal Q\in\mathbf{qo}(X)$\,.
\begin{enumerate}
\item[(i)] 
We say that $x$ is {\it generic} (with respect to $\mathcal Q$) and write $x\in\mathcal Q^{\rm g}$ if $x$ generates $\mathcal Q$\,, i.e. $\mathcal Q_x=\mathcal Q$\,.
\item[(ii)]
In the opposite case we say that $x$ is {\it non-generic} (with respect to $\mathcal Q$) and write $x\in\mathcal Q^{\rm n}$.
\item[(iii)] The orbit closure $\mathcal Q$ is said to be $\Xi${\it -minimal} if it does not contain any non-trivial invariant closed subset.
\end{enumerate}
\end{defn}

The next result follows immediately from the definitions:

\begin{lem}\label{intamplator}
One has $\mathcal Q^{\rm g}=\!\!\underset{x\in\mathcal Q^{\rm g}}{\bigcup}\mathcal O_x$ and $\,\mathcal Q^{\rm n}=\!\!\!\!\underset{\mathcal Q'\in\mathbf{qo}_0(\mathcal Q)}{\bigcup}\!\mathcal Q'$\,. 

The decomposition $\mathcal Q=\mathcal Q^{\rm g}\sqcup\mathcal Q^{\rm n}$ is $\Xi$-invariant and $\mathcal Q^{\rm g}$ is dense in $\mathcal Q$\,.
\end{lem}

\begin{cor}\label{castrav}
Let $\mathcal Q$ be a given orbit closure. Then $\,\Si(H_x)$ does not depend on $x\in\mathcal Q^{\rm g}$. In particular, if $\mathcal Q$ is $\Xi$-minimal, then $\big\{\Si\big(H_x\big)\!\mid\!x\in\mathcal Q\big\}$ is constant. 
\end{cor}

\begin{proof}
The first assertion is a consequence of Theorem \ref{operatorizat} (ii). Then the second one follows, since one has $\mathcal Q^{\rm n}=\emptyset$ precisely when $\mathcal Q$ is $\Xi$-minimal.
\end{proof}

If one aims only at proving Corollary \ref{castrav}, concentrating only on what happens on the generic points, there is a more direct approach (but using the same ideas). Let us assume (to simplify notations) that the groupoid $\Xi$ is {\it topologically transitive}, i.e. that there is at least one dense orbit. Thus there is an orbit closure $\mathcal Q=X$. We sketch below the proof of the following result:

\begin{prop}\label{strofok}
For a unit $x\in X$ the following assertions are equivalent:
\begin{enumerate}
\item[(i)]
The point $x$ is generic, meaning that its orbit is dense.
\item[(ii)]
The representation $\Pi_x$ is faithful.
\item[(iii)]
For any normal element $f$ of $\,{\sf C}_{\rm r}^*(\Xi,\o)$ one has $\,{\sf sp}[\Pi_x(f)]=\,{\sf sp}(f)$\,.
\end{enumerate}
\end{prop}

The minimal case corresponds to asking the three properties for every $x\in X$; in particular constancy of the spectrum follows. Particular cases of this proposition are in \cite[Th.\,4.3]{LS1} and \cite[Th.\,3.6.8]{Be}.

\begin{proof}
The implication (i) $\Rightarrow$ (ii) is our Proposition \ref{construint}. Concerning the converse, in \cite[Lemma\,2.6]{BB} the untwisted case is treated, but the proof stands as it is if the $2$-cocycle is present.

\smallskip
The equivalence between (ii) and (iii) is simple. On one hand, an injective morphism preserves the spectrum. In the opposite direction, use the fact that for normal elements the spectral radius coincides with the norm, showing that $\Pi_x$ is isometric.
\end{proof}

\section{Pseudo-differential operators associated with twisted dynamical systems}\label{goforit}

Recall that the locally compact group $\G$\,, with full group $C^*$-algebra ${\sf C}^*(\G)$\,,  is called {\it type $I$} if for every irreducible unitary representation $\pi:\G\to\mathbb B(\H_\pi)$\,, all the compact operators on the representation Hilbert space $\H_\pi$ are contained in the $C^*$-algebra $\widetilde\pi\big[{\sf C}^*(\G)\big]$\,. Here $\widetilde\pi$ is the integrated form of $\pi$\,, acting on $L^1(\G)\subset{\sf C}^*(\G)$ by $\widetilde\pi(u):=\int_\G u(a)\pi(a)da$\,. Other characterizations are available \cite{Di}. The group is unimodular if the left Haar measures are also right invariant.

\begin{defn}\label{admisibil}
We call {\it admissible group} a Hausdorff, second countable, unimodular, type $I$ locally compact group with unit $\e$ and Haar measure $d\m(a)\equiv da$\,.
\end{defn}

It will be clear that, for some purposes, only part the assumptions will be needed. Many classes of groups are admissible, as Abelian, compact, exponential (including the nilpotent ones), many types of solvable, connected real algebraic and semi-simple Lie groups. A discrete group is type $I$ if and only if it is the finite extension of an Abelian normal subgroup. 

\smallskip
The definition of a type $I$ group is less important than its main consequence, the existence of a Plancherel measure and a Fourier transformation, that we now briefly review under the full assumption that $\G$ is admissible. Full details can be found in \cite{Di}; see also \cite{MR}. 
We set $\widehat\G$ for {\it the unitary dual of} $\,\G$\,; by definition it is composed of unitary equivalence classes of strongly continuous irreducible Hilbert space unitary representations. There is a standard Borel structure and an important measure $\widehat\m$ (called {\it the Plancherel measure}, unique up to a constant) on $\widehat\G$\,. 
One can choose a $\widehat\m$-measurable field $\big\{\H_\vartheta\!\mid\! \vartheta\in\widehat\G\big\}$ of Hilbert spaces and a measurable section $\widehat\G\ni\vartheta\to\pi_\vartheta$ such that each $\pi_\vartheta$ is an irreducible representation in $\H_\vartheta$ belonging to the class $\vartheta$\,. Sometimes, by abuse of notation, instead of $\pi_\vartheta$ we will write $\vartheta$, identifying irreducible representations (of the measurable choice) with elements of $\widehat\G$\,.

\smallskip
Let us denote by $\mathscr B(\widehat\G):=\int_{\widehat\G}\mathbb B(\H_\vartheta)d\widehat\m(\vartheta)$ the direct integral von Neumann algebra and by $\mathscr B^2(\widehat\G):=\int_{\widehat\G}\mathbb B^2(\H_\vartheta)d\widehat\m(\vartheta)$ the Hilbert space direct integral of Hilbert-Schmidt operator spaces over the base $\big(\widehat\G,\widehat\m\big)$\,, with the usual scalar product
$$
\<\psi_1,\psi_2\>_{\!\mathscr B^2(\widehat\G)}:=\int_{\widehat\G}{\rm Tr}_\vartheta\big[\psi_1(\vartheta)\psi_2(\vartheta)^*\big]d\widehat\m(\vartheta)\,.
$$
{\it The Fourier transform}, basically defined by
\begin{equation*}\label{fourier}
[{\sf F}(u)](\vartheta)\equiv\widetilde\pi_\vartheta(u):=\int_\G u(a)\pi_\vartheta(a)^*da\,,
\end{equation*}
is an injective linear contraction $:L^1(\G)\to\mathscr B(\widehat\G)$ and defines a unitary map $:L^2(\G)\to\mathscr B^2(\widehat\G)$\,. On suitable subspaces of $\mathscr B^2(\widehat\G)$\,, the explicit form of the inverse reads
\begin{equation}\label{fourierinv}
\big[{\sf F}^{-1}(\psi)\big](a)=\int_{\widehat\G}{\rm Tr}_\vartheta[\psi(\vartheta)\pi_\vartheta(a)]d\widehat\m(\vartheta)\,.
\end{equation}

\smallskip
These preparations would be enough to introduce the pseudo-differential calculus of \cite{MR}. But we are interested here in its twisted generalization studied in \cite{BM}, and even in an extension to dynamical systems, generalizing the treatment of \cite{BLM} (see also \cite[Sect.\,7.4]{MR}). So there is one more ingredient to introduce.

\smallskip
Let $(\th,\O)$ be {\it a twisted action} of the admissible group $\G$ on the Hausdorff locally compact $\si$-compact space $X$. This means that $\th:X\times\G\to X$ is {\it a continuous action} to the left $(\th_b\circ\th_a=\th_{ba}\,, \forall\,a,b\in\G)$ and 
\begin{equation}\label{savorbim}
\O:\G\times\G\to C(X;\T):=\{\phi:X\to\T\mid \phi\ {\rm is\ continuous}\}
\end{equation}
is a continuous $2$-cocycle of the group opposite to $\G$\,, for which the composition is $(a,b)\to ba$\,, with values in the unitary group $C(X;\T)$ of the multiplier algebra of the $C^*$-algebra $C_0(X)$\,. With the notation $\Th_a(\phi):=\phi\circ\th_a$ (leading to a right action)\,, by definition, $\O$ is required to satisfy
\begin{equation}\label{statistic}
\O(a,b)\O(ba,c)=\Th_a[\O(b,c)]\O(a,cb)\,,\quad\forall\,a,b,c\in\G\,,
\end{equation}
$$
\O(a,{\rm e})={\bf 1}=\O({\rm e},a)\,,\quad\forall\,a\in\G\,.
$$
The {\it $2$-coboundaries} are here defined as 
\begin{equation}\label{coboundari}
\O(a,b)\equiv[\delta^1\!(\Lambda)](a,b):=\Lambda(a)\Th_a[\Lambda(b)]\Lambda(ba)^{-1},
\end{equation} 
for some continuous function $\Lambda:\G\to C(X;\T)$\,.

\smallskip
Let us denote by $C_0(X)\overline\otimes L^1(\G)$ the completed projective tensor product of the Banach spaces $C_0(X)$ and $L^1(\G)$ and by 
$$
\mathbf F\equiv{\rm id}\overline\otimes{\sf F}:C_0(X)\overline\otimes L^1(\G)\to C_0(X)\overline\otimes \mathscr B(\widehat\G)
$$ 
the partial Fourier transformation in the second variable.

\begin{defn}\label{biere}
The {\it twisted pseudo-differential operator} associated to the symbol $\Phi\in\mathbf F\big[C_0(X)\overline\otimes L^1(\G)\big]$ and to the point $x\in X$ acts on $u\in L^2(\G)$ as
\begin{equation}\label{bere}
\big[{\sf Op}_x(\Phi)u\big](a):=\int_\G\!\int_{\widehat\G}\O_x\big(b,ab^{-1}\big){\rm Tr}_\vartheta\big[\Phi\big(\th_b(x),\vartheta\big)\pi_\vartheta(ab^{-1})\big]u(b)d\widehat m(\vartheta)db\,.
\end{equation}
\end{defn}

We send to \cite{BM,MR} for technical details, properties and extensions. See also Remarks \ref{dumireste} and \ref{diverse} and Lemma \ref{vaevien}. There are connections with the well-developped theory of pseudodifferential calculus on Lie groupoids. Indeed, in the smooth case and when $\Omega\equiv 1$\,, (4.5) is the formula for a negative-order pseudodifferential operators on the transformation groupoid $X\rtimes \G$\,, with (reduced) Schwartz kernel $k(x,a):={\rm Tr}_\vartheta\big[\Phi(x,\vartheta)\pi_\vartheta(a)\big]$\,. We cannot develop this here; see \cite{ALN,LMN,NWX} and references therein.

\begin{rem}\label{diversse}
{\rm Formula \eqref{bere} would look more familiar for Abelian $\G$\,, for which the Fourier theory is simpler: $\widehat\G$ is a group (the Pontryagin dual), with Haar measure $d\widehat\m(\vartheta)\equiv d\vartheta$\,, and the irreducible representations are all $1$-dimensional (characters). In particular the symbol $\Phi$ is scalar-valued, being simply a complex function on the "phase-space" $\G\times\widehat\G$\,, so the trace in \eqref{bere} is spurious. One gets
\begin{equation*}\label{bereta}
\big[{\sf Op}_x(\Phi)u\big](a)=\int_\G\!\int_{\widehat\G}\O_x\big(b,a-b\big)\vartheta(a-b)\Phi\big(\th_b(x),\vartheta\big)u(b)d\vartheta db\,.
\end{equation*}
Let us forget the cocycle for a moment, assuming that $\O\equiv 1$\,. If $\G=\R^n$, then $\widehat\G$ identifies to $\R^n$ and $\vartheta(a-b)$ reads $e^{i(a-b)\vartheta}$. If, in addition, $X=\G=\R^n$ and $\th$ is the (free, transitive) action of $\R^n$ on itself by translations, then the operators $\big\{{\sf Op}_x(\Phi)\mid x\in\R^n\big\}$ are mutually unitarily equivalent and ${\sf Op}_0$ is the usual (right) pseudo-differential quantization. A general version of the $\tau$-quantization (covering the Kohn-Nirenberg and the Weyl calculus) is possible \cite{MR}, but we don't use it here. 
}
\end{rem}

Let us mention the notions and notations we use in connection with the dynamical system $(X,\th,\G)$\,. Remark \ref{primuta} will be relevant to relate these standard notions with those given before in the groupoid context. The {\it orbit} of a point $x\in X$ is denoted by ${\sf O}_x\!:=\th_\G(x)$ and its {\it orbit closure} ${\sf Q}_x\!:=\overline{{\sf O}_x}$ is the closure of the orbit. We write $x\overset{\bullet}{\sim}y$ if ${\sf Q}_x={\sf Q}_y$\,. A given orbit closure ${\sf Q}$ can be decomposed as ${\sf Q}^{\sf g}\sqcup{\sf Q}^{\sf n}$, where $y\in{\sf Q}^{\sf g}$ ({\it generic point}) if ${\sf Q}_y={\sf Q}$ and $z\in{\sf Q}^{\sf n}$ ({\it non-generic point}) if ${\sf Q}_z$ is strictly contained in ${\sf Q}$\,.  The proof of the next result will be given later, after suitable preparations.

\begin{thm}\label{fiveoclock}
Suppose that $\Si$ is a $C^*$-algebraic spectral set. Let $\G$ be an admissible group and $(\th,\O)$ a twisted action of $\,\G$ on the Hausdorff locally compact space $X$. Let $\Phi\in\mathbf F\big[C_0(X)\overline\otimes L^1(\G)\big]$\,.
\begin{enumerate}
\item[(i)]
If $\,x\overset{\bullet}{\sim}y$\,, then $\Si\big[{\sf Op}_x(\Phi)\big]=\Si\big[{\sf Op}_y(\Phi)\big]$\,.
\item[(ii)]
In particular if $\,{\sf Q}\subset X$ is a minimal orbit closure, then the operators $\big\{{\sf Op}_x(\Phi)\mid x\in{\sf Q}\big\}$ have the same spectrum, the same norm, the same $\epsilon$-pseudospectrum and the same numerical range.
\end{enumerate}
\end{thm}

The assumptions of the theorem will hold true below. We first put the setting above in a groupoid perspective and then we will deduce the theorem from our previous results.
Define {\it the transformation groupoid} $\,\Xi:=X\!\rtimes_\th\G$\,, equal to $X\!\times\!\G$ as a set, with the product topology and the structure mappings
\begin{equation}\label{fitze}
\begin{aligned}
{\rm d}(x,a):=x\,,&\quad {\rm r}(x,a):=\th_a(x)\,,\\
(x,a)^{-1}\!:=\big(\th_a(x),a^{-1}\big)\,,&\quad \big(\th_a(x),b\big)(x,a):=(x,ba)\,.
\end{aligned}
\end{equation}
Thus the fibers are
$$
\Xi_x=\{x\}\times\G\,,\quad\Xi^x=\big\{\big(\th_{a^{-1}}(x),a\big)\big\vert\,a\in\G\big\}\cong\G
$$
and $\Xi^x_x$ may be identified with the isotropy group of $x$ under the action $\th$\,. 

\begin{rem}\label{primuta}
The transformation groupoid $X\!\rtimes_\th\G$ is Hausdorff, locally compact and possesses a Haar system. The standard notions associated to the dynamical system $(X,\th,\G)$ correspond to those of the transformation groupoid: For $x,y\in X$ one has ${\sf O}_x=\mathcal O_x$\,, ${\sf Q}_x=\mathcal Q_x$\,, ${\sf Q}^{\rm g}_x=\mathcal Q^{\rm g}_x$\,, ${\sf Q}^{\rm n}_x=\mathcal Q^{\rm n}_x$ and $x\overset{\bullet}{\sim} y$ if and only if $x\sim y$\,; this follows easily from the explicit form of ${\rm d}$ and ${\rm r}$\,. 
\end{rem}

The next result about $2$-cocycles may be found in \cite{Gi}; for the convenience of the reader, and because our conventions are different from those of \cite{Gi}, we indicate a proof. In Example \ref{inverse} one may find an example of physical interest.

\begin{lem}\label{zaconecsan}
The formula 
\begin{equation*}\label{obeg}
\o\big((\th_a(x),b),(x,a)\big):=[\O(a,b)](x)\equiv\O_x(a,b)
\end{equation*}
provides a one-to-one correspondence between group $2$-cocycles \eqref{savorbim} associated to the action $\th$ and groupoid $2$-cocycles $\o$ of the transformation groupoid $\,\Xi:=X\!\rtimes_\th\!\G$\,. The correspondence $\o\leftrightarrow\O$ preserves cohomology (coboundaries correspond to coboundaries).
\end{lem}

\begin{proof}
The verifications are straightforward. For example, assuming that $(\th,\O)$ is a twisted action of $\G$ on $X$, one computes for three composable elements $\xi:=\big(\th_{ba}(x),c\big),\eta:=\big(\th_a(x),b\big),\zeta:=(x,a)$
$$
\begin{aligned}
\o(\xi,\eta)\,\o(\xi\eta,\zeta)&=\o\big(\big(\th_{ba}(x),c\big),\big(\th_a(x),b)\big)\,\o\big((\th_a(x),cb),(x,a)\big)\\
&=[\O(b,c)]\big[\th_a(x)\big]\,[\O(a,cb)](x)\\
&=\big(\Th_a[\O(b,c)]\O(a,cb)\big)(x)\\
&=\big[(\O(a,b)\O(ba,c)\big](x)\\
&=[\O(a,b)](x)\,[\O(ba,c)](x)\\
&=\o\big((\th_a(x),b),(x,a)\big)\,\o\big(\big(\th_{ba}(x),c\big),(x,ba)\big)\\
&=\o(\eta,\zeta)\,\o(\xi,\eta\zeta)\,.
\end{aligned}
$$
The fourth equality is \eqref{statistic}; the others are just consequences of the definitions. 

\smallskip
On the other hand, if $\O=\delta^1(\Si)$\,, as in \eqref{coboundari}, one sets $\si=\Xi\to\mathbb T\,,\,\si(x,a):=[\Si(a)](x)$ and then
$$
\begin{aligned}
\o(\eta,\zeta)&=\o\big((\th_a(x),b),(x,a)\big)=[\delta^1(\Si)(a,b)](x)\\
&=\big(\Si(a)\Th_a[\Si(b)]\Si(ba)^{-1}\big)(x)\\
&=[\Si(a)](x)\,[\Si(b)]\big(\th_a(x)\big)\,[\Si(ba)](x)^{-1}\\
&=\si(\zeta)\si(\eta)\si(\eta\zeta)^{-1},
\end{aligned}
$$
so $\o=\delta^1(\si)$ is a groupoid $2$-coboundary.
\end{proof}

\begin{rem}\label{dumireste}
{\rm It follows easily \cite{Gi} that one can identify the reduced twisted groupoid $C^*$-algebra ${\sf C}^*_{\rm r}\big(X\!\rtimes_\th\!\G,\o\big)$ with {\it the reduced twisted crossed product} \cite{PR1} $C_0(X)\!\rtimes_{\Th,{\rm r}}^{\Omega}\G$\,. This makes the magnetic pseudo-differential calculus in $\R^n$ \cite{IMP,MP,MPR2} part of the twisted groupoid framework. The $C^*$-algebra $C_0(X)\!\rtimes_{\Th,{\rm r}}^\Omega\!\G$ is a suitable closure of the Banach $^*$-algebra $C_0(X)\overline\otimes L^1(\G)$\,, with a certain $^*$-algebraic structure defined by the twisted action $(\Th,\O)$\,.}
\end{rem}

\begin{lem}\label{vaevien}
For every $x\in X$ one has ${\sf Op}_x(\Phi)=\Pi_x\big[\big({\rm id}\otimes{\sf F}^{-1}\big)(\Phi)\big]$\,.
\end{lem}

\begin{proof}
In terms of the twisted action $(\th,\O)$ the twisted groupoid composition can be written
\begin{equation*}\label{interms}
(f\star_\o\!g)(x,a)=\int_\G f\big(\th_b(x),ab^{-1}\big)g(x,b)\O_x(b,ab^{-1})db\,.
\end{equation*}
For fixed $x$ there is an obvious unitary identification 
$$
L^2(\Xi_x;\lambda_x)=L^2\big(\{x\}\times\G;\delta_x\times\m\big)\ni\tilde u\rightarrow u\in L^2(\G)\,,\quad u(a):=\tilde u(x,a)\,,
$$
under which the regular representation associated to $x$ reads
\begin{equation}\label{regrep}
\big[\Pi_x(f)u\big](a)=\int_\G f\big(\th_b(x),ab^{-1}\big)u(b)\O_x(b,ab^{-1})db\,.
\end{equation}
A priori \eqref{regrep} makes sense as it is on $C_{\rm c}(X\!\times\!\G)$ and extends by continuity to ${\sf C}^*_{\rm r}(X\!\rtimes_\th\!\G,\o)$\,. By Remark \ref{dumireste} 
$$
C_{\rm c}(X\!\times\!\G)\subset C_0(X)\overline\otimes L^1(\G)\subset{\sf C}^*_{\rm r}(X\!\rtimes_\th\!\G,\o)
$$
and \eqref{regrep} also makes sense on $C_0(X)\overline\otimes L^1(\G)$\,. These and the Fourier inversion formula \eqref{fourierinv} imply the result. In \eqref{bere} the right-hand-side should be interpreted as an iterated integral, the integration in $\widehat\G$ being done first.
\end{proof}

{\it Proof of Theorem \ref{fiveoclock}.} The Lemmas above established a full dictionary between the twisted groupoid and the twisted dynamical system settings. Then the theorem follows from the results \ref{operatorizat}, \ref{castrav} and from the examples of Section \ref{menghina}. \qed

\begin{rem}\label{potential}
{\rm For simplicity, in Theorem \ref{fiveoclock}, we only made use of the twisted groupoid algebra (or, equivalently, of the twisted crossed product). General elements of the multiplier algebras, in this case, are not easy to figure out. But a particular case is easily accessible. Suppose that a symbol $\Psi$ only depends on the first variable, i.e. $\Psi(x,\vartheta):=V(x)$ for some bounded continuous function $V:X\to \mathbb C$ (a "potential"). Then ${\sf Op}_x(V)$ is the operator in $L^2(\G)$ of multiplication with the function $\G\ni a\to V[\th_a(x)]\in\mathbb C$ and it can be added in the theorem to ${\sf Op}_x(\Phi)$\,. The groupoid interpretation yields from the fact that the $C^*$-algebra of all bounded continuous complex functions on $X$ embeds naturally as multipliers of the twisted groupoid $C^*$-algebra, as it follows from \cite{Re}[II,\,Prop.\,2.4(ii)] for instance. 
}
\end{rem}

\begin{rem}\label{diverse}
{\rm The twisted pseudo-differential quantization in the $\R^n$-case is mainly motivated by the quantum theory of systems placed in a variable magnetic field and is treated in \cite{BLM,IMP,MP,MPR2}.  In this case the cocycles are imaginary exponentials of the fluxes of the magnetic field through suitable triangles. Even in this very particular case the full Theorem \ref{fiveoclock} is new, especially because of the generality of the spectral set function $\Si$\,. 
}
\end{rem}

\begin{ex}\label{inverse}
{\rm Let us indicate briefly a situation involving magnetic fields \cite{BM,MN1}, that is more general than the one of Remark \ref{diverse}. Keeping the setting of this section, assume in addition that $\G$ is a connected, simply connected and nilpotent Lie group; then it is admissible. Its Lie algebra $\mathfrak g$ is diffeomorphic to $\G$ through the exponential map; its inverse is denoted by $\log:=\exp^{-1}\!:\G\to\mathfrak g$\,. Let $B$ be {\it a magnetic field}, i.e. a closed $2$-form on $\G$\,. It can be seen as a smooth map associating to any $a\in\G$ the skew-symmetric bilinear form $B(a):\mathfrak g\times\mathfrak g\to\R$\,. To make $B$ compatible with the space $X$ on which $\G$ is supposed to act, 
we assume that $\G$ is a dense, open subset of the compact space $X$ and that the action of $\G$ on itself by left translations extends to a continuous action $\th$ of $\G$ on $X$\,. 
The condition on $B$ is that it extends continuously to $X$. Set
\begin{equation*}\label{nothard}
\Gamma\!_B(x;a,b):=\int_0^1\!\int_0^t\! B\big(\th_{\exp[s\log(ba)+(t-s)\log a]}(x)\big)\big(\log a,\log(ba)\big) dsdt\,.
\end{equation*}
It is the flux of the magnetic field $B_x(\cdot):=B\big(\th_\cdot(x)\big)$ through a "triangle" in $\G$ determined by the points $\e,a,ba$\,, defined as the image through $\exp$ of the triangle of vertices $0,\log a,\log(ba)$ in the Lie algebra. Using the definitions and Stokes' formula, one shows that
\begin{equation*}\label{zdricartz}
\o\big((\th_a(x),b),(x,a)\big)\equiv\O_x(a,b):=e^{i\Gamma\!_B(x;a,b)},\quad a,b\in\G\,,\ x\in X
\end{equation*}
defines a $2$-cocycle on the transformation groupoid $X\!\rtimes_\th\!\G$\,, and one can apply Theorem \ref{fiveoclock}.
}
\end{ex}

\begin{ex}\label{averse}
{\rm Suppose that $\G$ is Abelian, as in Remark \ref{diversse}, and denote by ${\rm AP}(\G)$ the $C^*$-algebra of all the complex continuous {\it almost periodic} functions on $\G$ (the set of all the translations of such a function is relatively compact in the uniform norm; equivalently, it can be uniformly approximated by trigonometric polynomials). It is invariant under translations. The spectrum of this $C^*$-algebra (cf. \cite[Sect.\,16]{Di} and \cite[Sect.\,26]{HR}) is called {\it the Bohr compactification of $\,\G$} and is denoted by ${\rm b}\G$\,. Using Pontryagin duality, it can be described as the dual of the dual of $\G$\,, this first dual being given the discrete topology: ${\rm b}\G=\widehat{\big(\widehat\G_d\big)}$\,. Thus its elements are group morphisms $x:\widehat\G\to\T$\,, continuous or not. Since $\G$ is isomorphic with the usual bidual $\widehat{\widehat\G}$ by $a\to\varpi(a)$\,, given by $[\varpi(a)](\vartheta):=\vartheta(a)$ for every $\vartheta\in\widehat\G$\,, our group $\G$ identifies to a dense subgroup of the Bohr group (but its topology is different from the induced topology). The continuous function $\varphi:\G\to\mathbb C$ is almost periodic if and only if it extends to a continuous function ${\rm b}\varphi:{\rm b}\G\to\mathbb C$\,. For us {\it ${\rm b}\G$ is relevant since it is minimal under the action ${\rm b}\G\times\G\ni(x,a)\to a+x\in{\rm b}\G$}\,; equivalently, ${\rm AP}(\G)$ is {\it $\G$-simple}, i.e.\! it has no proper $\G$-invariant closed ideal. Therefore, stated loosely, {\it if $\,\Si$ is a $C^*$-algebraic spectral set, $\Phi$ has almost periodic coefficients and $\G\ni c\to\Omega_c(a,b)\in\T$ is almost periodic, then $\Si\big[{\sf Op}_x(\Phi)\big]$ actually does not depend on $x\in{\rm b}\G$}\,. The reader will easily make this statement precise, as a consequence of Theorem \ref{fiveoclock} (b) and the discussion above. One more remark in this setting: if $\Omega$ is trivial and $\Phi$ only depends on the variable $\vartheta$\,, being the Fourier transform of a function $\nu\in L^1(\G)$\,, it is easy to see that ${\sf Op}_x(\Phi)$ is just the operator ${\sf Conv}_\nu$ of convolution with $\nu$ in $L^2(\G)$\,. It does not depend  on $x$\,, but one may look (say) at the perturbation $\Phi+V$\,, cf. Remark \ref{potential}, with $V:\G\to\mathbb C$ continuous and almost periodic. We deduce that {\it $\Si\big({\sf Conv}_\nu+V_x\big)$ is independent of $x$}\,, where $V_x$ denotes the operator of multiplication with the function $\G\ni a\to ({\rm b}V)(a+x)\in\mathbb C$\,, in terms of the unique extension of $V$ to the Bohr compactification, in which the sum $a+x$ is performed.}
\end{ex}

\begin{rem}\label{minimalisim}
{\rm Minimality for locally compact dynamical systems, relevant in Theorem \ref{fiveoclock}, is an amazingly rich and complex phenomenon, going much beyond almost periodicity. It encompasses distal and almost automorphic behavior, and not only. It is not the right place here to make a summary of this topic. We refer to \cite{Au,EE} for comprehensive treatments and especially to \cite{Kn} for a more condensed reference presenting a unified picture. We only provide the reader with the following comparison between almost periodic and the much more general minimal functions (the context is once again non-commutative): The continuous function $\varphi:\G\to\mathbb C$ is almost periodic if and only if for every $\epsilon>0$\,, the set of all its $\epsilon$-quasi-periods 
$$
P_\epsilon(\varphi;\G):=\{b\in\G\mid |\varphi(ba)-\varphi(a)|<\epsilon\,,\ \forall\,a\in\G\}
$$ 
is relatively dense, i.e. one has $P_\epsilon(\varphi;\G)K=\G$ for some compact set $K\subset\G$\,. On the other hand, one calls the function $\varphi$ {\it minimal} if only 
$$
P_\epsilon(\varphi;F):=\{b\in\G\mid |\varphi(ba)-\varphi(a)|<\epsilon\,,\ \forall\,a\in F\}
$$
is relatively dense for every $\epsilon$ and {\rm finite} subset $F\in\G$\,. The interest relies on the fact that {\it $\varphi$ is minimal if and only if the smallest left invariant $C^*$-algebra containing $\varphi$ and composed of bounded continuous functions on $\G$\,, denoted by $\A_\varphi$\,, is $\G$-simple, and this happens exactly when its Gelfand spectrum $X={\rm Sp}(\A_\varphi)$ is a minimal dynamical system}. This allows applying our result on the constancy of the $C^*$algebraic spectral sets to pseudodifferential operators "with minimal coefficients and cocycles".
}
\end{rem}

\section{Groupoids and their non-invariant restrictions}\label{grafitoni}

We start by examining a suitable type of groupoids and their restrictions to certain {\it non-invariant} unit subspaces. This prepares the ground for the abstract results and the examples treated in the following sections. Non-invariance is an essential issue for applications; this brings in some complications. In particular, the problem of restricting the Haar system is not trivial, and this will lead to  Definition \ref{tame}.

\begin{rem}\label{aculici}
For simplicity, we decided not to include the $2$-cocycle $\o$\,, but the interested reader will easily supply the modifications needed for the twisted case. Another reason for skipping $\o$ is that in our second example, treated in Section \ref{germanofoni}, the cohomology of the coarse groupoid $\Xi(X_0,d)$ associated to a uniformly discrete metric space $(X_0,d)$ with bounded geometry is anyway trivial. The example of Section \ref{gerfomoni} can be twisted in a non-trivial way, in terms of twisted partial actions \cite{Ab}. 
\end{rem}

\smallskip
Recall that for any locally closed subset $M$ of $X$ (in particular if $M$ is open or closed)\,, one defines  {\it the restricted groupoid} 
$$
\Xi(M)\!:=\Xi_{M}\cap\Xi^M=\{\xi\in\Xi\mid \d(\xi),\r(\xi)\in M\}\,.
$$ 
If $M$ is invariant, then $\Xi(M)=\Xi_M=\Xi^M$ and this restriction is called {\it a reduction}. Restrictions which are not reductions will appear often subsequently.
Non-invariant restrictions of Haar systems will be needed below and are not a trivial matter. We mention that the terminology reduction - restriction is not standard in the literature.

\begin{defn}\label{tame}
Let $\big\{\lambda_x\mid x\in X\big\}$ be a right Haar system on $\Xi$\,. We say that the subset $M\subset X$ is {\it tame} if it is closed and 
\begin{equation}\label{sarcoidoza}
\lambda(M)_x:=\lambda_x|_{\Xi(M)_x}\,,\quad x\in M
\end{equation}
 defines a right Haar  system for the restricted groupoid $\Xi(M)$\,.
\end{defn}

Right invariance of \eqref{sarcoidoza} follows from the right invariance of the initial Haar system. The issues are the continuity and the full support condition.

\begin{lem}\label{sasperam}
Invariant closed subsets or clopen subsets are tame.
\end{lem}

\begin{proof}
This is well-known for invariant closed subsets. 

\smallskip
Suppose now that the closed subset $M$ is also open.  Since $\Xi(M)$ is open, any compactly supported continuous function $\varphi$ on $\Xi( M)$ may be identified with a compactly supported continuous function on $\Xi$ (extending by null values), and then the continuity of the map
$$
M\ni y\to\int_{\Xi(M)}\!\varphi(\xi)d\lambda(M)_y(\xi)=\int_{\Xi}\varphi(\xi)d\lambda_y(\xi)\in\mathbb C
$$ 
is clear. One still has to show that the support of $\lambda(M)_y$ is the entire $\Xi(M)_y$\,. Note that $\Xi(M)_y=\Xi(M)\cap\Xi_y$ is open in $\Xi_y$\,. If $A$ is open in $\Xi(M)_y$ it will also be open in $\Xi_y$\,, hence $\lambda_y(A)=\lambda(M)_y(A)=0$ is equivalent to $A=\emptyset$\,, since $\lambda_y$ is a full measure on $\Xi_y$\,.
\end{proof}

\smallskip
The following framework will be used several times below. By {\it amenable groupoid}, we mean "topologically amenable", cf \cite{ADR}. In \cite{Re1} this is shown to be equivalent to Borel amenability. These concepts are too intricate to be reviewed here; what we really need is that amenability, for a $\si$-compact locally compact groupoid with Haar system, implies that the full and the reduced groupoid algebras coincide \cite{Re1}.

\begin{Setting}\label{orange}
{\rm 
\begin{enumerate}
\item[(i)]
Let $\Xi$ be a Hausdorff amenable  locally compact $\si$-compact groupoid over a unit space $X$, with right Haar system $\big\{\lambda_x\!\mid \!x\in X\big\}$. Let $X_0\subset X$ be open and invariant. The points of $ X$ will be denoted by $x,y,z,\dots$\,, or $a,b,c,\dots$ when we are sure that they belong to $X_0$\,. The reductions to $X_0$ and $X_\infty:=X\setminus X_0$ (closed and invariant) are denoted, respectively, by $\Xi(X_0)$ and $\Xi(X_\infty)$\,. 
\item[(ii)]
For a given tame subset $Y$ of $X$ containing $X_\infty$\,, we put $Y_0\!:=X_0\cap Y$  (thus $Y=Y_0\sqcup X_\infty$)\,. Thus we also have the (non-invariant) restricted groupoid $\,\Xi(Y):=\{\xi\in\Xi\!\mid\!{\rm d}(\xi),{\rm r}(\xi)\in Y\}$\,.
\end{enumerate}
}
\end{Setting}

\begin{rem}\label{aceltropic}
{\rm  It is easy to see that $\Xi(Y)$ is a locally closed subgroupoid of $\Xi$\,; then, by \cite[Prop.\,5.1.1]{ADR}, it is amenable. The groupoid $\Xi(Y)$ comes with two extra reductions $\,[\Xi(Y)](Y_0)$ and $\,[\Xi(Y)](X_\infty)$\,. However it is follows from the definitions that $\,[\Xi(Y)](X_\infty)=\,\Xi(X_\infty)$\,, and this will be important. It is clear, by our assumptions, that each appearing groupoid is amenable.
}
\end{rem}

Let us denote by $\Xi_\alpha\equiv\Xi(X_\alpha)$ any of the groupoids introduced above, over the base space $X_\alpha$\,, with $X_\alpha=X_0,Y_0,X_\infty,\!Y,X$. In each case, the right Haar system $\big\{\lambda_{\alpha,x}:=\lambda(X_\alpha)_x\!\mid \!x\in X_\alpha\big\}$ is given, or defined by restriction, by tameness. For the sake of notations, recall that the vector space $C_{\sf c}(\Xi_\alpha)$ of all continuous compactly supported complex functions on $\Xi_\alpha$ is a $^*$-algebra under the involution $f^{\star_\alpha}(\xi):=\overline{f\big(\xi^{-1}\big)}$ and the multiplication
\begin{equation*}\label{multilaw}
(f\star_\alpha g)(\xi):=\int_{\Xi_\alpha}\!f(\eta^{-1})g(\eta\xi)d\lambda_{\alpha,\r(\xi)}(\eta)=\int_{\Xi_\alpha}\!f(\xi\eta^{-1})g(\eta)d\lambda_{\alpha,\d(\xi)}(\eta)\,.
\end{equation*}
Then the groupoid $C^*$-algebra ${\sf C}^*\big(\Xi_\alpha\big)$ (coinciding here with the reduced one ${\sf C}^*_{\sf r}\big(\Xi_\alpha\big)$\,, by amenability) is the completion of the $^*$-algebra $\big(C_{\sf c}(\Xi_\alpha),\star_\alpha,^{\star_\alpha}\!\big)$ with respect to the (reduced) norm
\begin{equation*}\label{iarba}
\p\!f\!\p_{X_\alpha}\,:=\,\sup_{x\in X_\alpha}\big\Vert\Pi_{\alpha,x}(f)\big\Vert_{\mathbb B(\H_{\alpha,x})}\,,
\end{equation*}
where $\H_{\alpha,x}$ is the Hilbert space $L^2\big(\Xi_{\alpha,x};\lambda_{\alpha,x}\big)$ and $\Pi_{\alpha,x}:C_{\sf c}(\Xi_\alpha)\to\mathbb B(\H_{\alpha,x})$ is the regular representation defined by $x$  through the formula
\begin{equation*}\label{iarbamate}
\Pi_{\alpha,x}(f)u:=f\star_\alpha\!u\,,\quad\forall\,f\in C_{\sf c}(\Xi_\alpha)\,,\;u\in\H_{\alpha,x}\,.
\end{equation*}

The next result will be basic; we only state it in the generality we need. The point (ii), treating invariant restrictions (i.\,e.\,reductions) is well-known \cite{MRW} for the full $C^*$-algebras, and recall that we work under amenability assumptions, so the full and the reduced algebras coincide. 
We provide a proof of point (i), since non-invariant restrictions are seldom considered in the literature and since formula \eqref{passion} will also be used in the proof of Theorem \ref{aguacate}. 

\begin{lem}\label{matraguna}
The indices $\alpha,\beta,\gamma$ label groupoids $\Xi_\alpha\supset\Xi_\beta\supset\Xi_\gamma$ as above, the smaller ones being obtained from the larger ones by restriction.
\begin{enumerate}
\item[(i)]
Let us define for $X_\beta$ tame subset of $\,X_\alpha$ the restriction map
\begin{equation*}\label{maselarita}
{\sf R}_{\alpha\beta}: C_{\rm c}\big(\Xi_\alpha\big)\to C_{\rm c}\big(\Xi_\beta\big)\,,\quad {\sf R}_{\alpha\beta}(f):=f|_{\Xi_\beta}.
\end{equation*}
It is linear and involutive and extends to a contraction of the groupoid $C^*$-algebras 
\begin{equation*}\label{musetel}
{\sf R}_{\alpha\beta}: {\sf C}^*\big(\Xi_\alpha\big)\to {\sf C}^*\big(\Xi_\beta\big)\,,
\end{equation*}
such that 
\begin{equation}\label{muschi}
{\sf R}_{\beta\gamma}\circ{\sf R}_{\alpha\beta}={\sf R}_{\alpha\gamma}\quad{\rm if}\ X_\alpha\supset X_\beta\supset X_\gamma\,.
\end{equation}
\item[(ii)]
If $\,X_\beta$ is closed in $X_\alpha$ and invariant under $\Xi_\alpha$\,, the map ${\sf R}_{\alpha\beta}: {\sf C}^*\big(\Xi_\alpha\big)\to {\sf C}^*\big(\Xi_\beta\big)$ is a $C^*$-algebraic morphism and one has the short exact sequence
\begin{equation*}\label{sash}
0\longrightarrow{\sf C}^*\big(\Xi_{X_\alpha\setminus X_\beta}\big)\overset{{\sf J}_{\alpha\beta}}{\longrightarrow}{\sf C}^*(\Xi_\alpha)\overset{{\sf R}_{\alpha\beta}}{\longrightarrow}{\sf C}^*(\Xi_\beta)\longrightarrow 0\,,
\end{equation*}
where ${\sf J}_{\alpha\beta}$ is the extension of the canonical injection of $\,C_{\rm c}\big(\Xi_{X_\alpha\setminus X_\beta}\big)$ into $C_{\rm c}(\Xi_\alpha)$\,.
\end{enumerate}
\end{lem}

\begin{proof}
We prove (i). Since $\Xi_\beta$ is closed, the restriction map is correctly defined at the level of continuous compactly supported functions. Both linearity and $^*$-preserving are clear. The relation \eqref{muschi} is also obvious on continuous functions with compact support.

\smallskip
 Contractivity (guaranteeing the extension and then the validity of \eqref{muschi}) follows from the form of the reduced norms and from the inclusion $X_\beta\subset X_\alpha$ if one shows for each $x\in X_\beta$ that
\begin{equation*}\label{mazga}
\big\Vert\Pi_{\beta,x}\big[{\sf R}_{\alpha\beta}(f)\big]\big\Vert_{\mathbb B(\H_{\beta,x})}\le\big\Vert\Pi_{\alpha,x}(f)\big\Vert_{\mathbb B(\H_{\alpha,x})}\,,\quad\forall\,f\in C_{\rm c}\big(\Xi_\alpha\big)\,.
\end{equation*}
As said before, we use on $\Xi_\beta$ the restriction of the right Haar measure on $\Xi_\alpha$\,.
Let us set $\Xi_x^{\alpha\beta}:=\Xi_{\alpha,x}\setminus\Xi_{\beta,x}$\,, so that one has $\Xi_{\alpha,x}=\Xi_{\beta,x}\sqcup\Xi_x^{\alpha\beta}$. Then, denoting by $\lambda_x^{\alpha\beta}$ the restriction of $\lambda_{\alpha,x}$ to the additional set $\Xi_x^{\alpha\beta}$, one has $\H_{\alpha,x}=\H_{\beta,x}\oplus\H_x^{\alpha\beta}$\,, where besides the previous notation $\H_{\gamma,x}:=L^2\big(\Xi_{\gamma,x};\lambda_{\gamma,x}\big)$\,, we introduced the Hilbert space $\H_x^{\alpha\beta}:=L^2\big(\Xi_x^{\alpha\beta};\lambda_x^{\alpha\beta}\big)$\,. There is a canonical isometric embedding $\H_{\beta,x}\ni v\to\tilde v\in\H_{\alpha,x}$ consisting in extending by zero values outside $\Xi_{\beta,x}$\,.
For every $v\in\H_{\beta,x}$ one has
$$
\big\Vert\Pi_{\beta,x}\big[{\sf R}_{\alpha\beta}(f)\big]v\big\Vert_{\H_{\beta,x}}^2\!=\big\Vert f|_{\Xi_{\beta}}\!\star_\beta\! v\big\Vert_{\H_{\beta,x}}^2
\!=\int_{\Xi_{\beta,x}}\big\vert\big(f|_{\Xi_{\beta}}\!\star_\beta\! v\big)(\xi)\big\vert^2 d\lambda_{\beta,x}(\xi)\,,
$$
while for every $u\in\H_{\alpha,x}$ one has
$$
\begin{aligned}
\big\Vert\Pi_{\alpha,x}(f)u\big\Vert_{\H_{\alpha,x}}^2&=\big\Vert f\star_\alpha\! u\big\Vert_{\H_{\alpha,x}}^2\\
&=\int_{\Xi_{\beta,x}}\big\vert(f\star_\alpha\! u)(\xi)\big\vert^2 d\lambda_{\beta,x}(\xi)+\int_{\Xi_x^{\alpha\beta}}\big\vert(f\star_\alpha\! u)(\xi)\big\vert^2 d\lambda_x^{\alpha\beta}(\xi)\,.
\end{aligned}
$$
Since the second term above is positive, and since the unit ball of $\H_{\beta,x}$ embeds isometrically in the unit ball of $\H_{\alpha,x}$\,, it is enough to show that, if $\p\!v\!\p_{\H_{\beta,x}}\,\le 1$\,, one has
\begin{equation}\label{necesitata}
\int_{\Xi_{\beta,x}}\!\big\vert\big(f|_{\Xi_{\beta}}\!\star_\beta\!v\big)(\xi)\big\vert^2 d\lambda_{\beta,x}(\xi)=\int_{\Xi_{\beta,x}}\!\big\vert\big(f\!\star_\alpha \!\tilde v\big)(\xi)\big\vert^2 d\lambda_{\beta,x}(\xi)\,.
\end{equation}
But for $\xi\in\Xi_{\beta,x}$\,, since $\tilde v$ is null on $\Xi_x^{\alpha\beta}$, we can write
\begin{equation}\label{passion}
\begin{aligned}
(f\star_\alpha\!\tilde v)(\xi)&=\int_{\Xi_{\alpha,x}}\!\!f\big(\xi\eta^{-1}\big)\tilde v(\eta) d\lambda_{\alpha,x}(\eta)\\
&=\int_{\Xi_{\beta,x}}\!\!f|_{\Xi_\beta}\!\big(\xi\eta^{-1}\big)v(\eta) d\lambda_{\beta,x}(\eta)=(f|_{\Xi_{\beta}}\!\star_\beta v\big)(\xi)\,;
\end{aligned}
\end{equation}
for the second equality we also used the fact that if $\xi,\eta\in\Xi_{\beta,x}$\,, then $\xi\eta^{-1}\!\in\Xi_\beta$\,. So in \eqref{necesitata} the two integrands are equal, and the proof is finished.
\end{proof}

\begin{rem}\label{licheni}
{\rm Multiplicativity is no longer true for ${\sf R}_{\alpha\beta}$ if $X_\beta$ is not invariant! When comparing $\big[{\sf R}_{\alpha\beta}(f\star_\alpha g)\big](\xi)$ with $\big[{\sf R}_{\alpha\beta}(f)\star_\beta{\sf R}_{\alpha\beta}(g)\big](\xi)$\,, there is an extra integration over $\Xi_{\alpha,x}\setminus\Xi_{\beta,x}$ (with $x:=r(\xi)\in X_\beta$) in the first expression. In particular, the restriction maps ${\sf C}^*(\Xi)\mapsto{\sf C}^*\big(\Xi_Y\big)$ and ${\sf C}^*\big(\Xi_X\big)\mapsto{\sf C}^*\big(\Xi_Y\big)$ are not $^*$-morphisms. If $X_\beta$ is invariant, one simply has $\Xi_{\alpha,x}\setminus\Xi_{\beta,x}=\emptyset$ for every $x\in X_\beta$ and multiplicativity is restored.
When $X_\beta$ is a clopen set, ${\sf C}^*_{\rm r}(\Xi_\beta)\subset{\sf C}^*_{\rm r}(\Xi_\alpha)$ and ${\sf R}_{\alpha\beta}$ satisfies the conditional expectation property of \cite[Def,\,1.3]{Ri}, meaning that
$$
{\sf R}_{\alpha\beta}(f\star g)={\sf R}_{\alpha\beta}(f)\star g\quad{\rm and}\quad{\sf R}_{\alpha\beta}(g\star f)=g\star {\sf R}_{\alpha\beta}(f)
$$
for any $f\in{\sf C}^*_{\rm r}(\Xi_\alpha)$ and $g\in{\sf C}^*_{\rm r}(\Xi_\beta)$\,. This is exactly the content of \eqref{passion} and the commutativity of the East diagram in \eqref{conopida}. We recall that in our case the full and the reduced groupoid $C^*$-algebras coincide.
}
\end{rem}

\section{The intrinsic Decomposition Principle}\label{grafomoni}

Besides Definition \ref{ierarhie}, we are also going to use

\begin{defn}\label{treclaesentiale}
Let $\Si$ be a $C^*$-algebraic spectral set. If $\mathscr K$ is a (closed $^*$-invariant) ideal of the unital $C^*$-algebra $\mathscr E$ and $\nu:\mathscr E\to\mathscr E/\mathscr K$ is the canonical quotient epimorphism, one defines {\it the $\mathscr K$-essential $\Si$-spectrum} through
\begin{equation}\label{esentialul}
\Si^\mathscr K _{\rm ess}(E\!\mid\!\E):=\Si\big(\nu(E)\!\mid\!\mathscr E/\mathscr K\big)\subset\Si(E\!\mid\!\E)\,.
\end{equation}
\end{defn}

\begin{rem}\label{derivat}
{\rm Of course, this is just a convenient derived object. In particular, if $\Si={\sf sp}$\,, $\E$ is a $C^*$-algebra of bounded linear operators in the Hilbert space $\H$ and $\mathcal K:=\E\cap\mathbb K(\H)$\,, one gets the usual essential spectrum. This one has other useful well-known descriptions, in terms of Fredholm properties, for instance. For self-adjoint operators it is the complement with respect to the spectrum of the family of isolated eigenvalues with finite multiplicity. Such a phenomenon may occur in other situations. Subsection 3.3. in \cite{MN1} refers to essential numerical ranges. Very many concepts that can go under the name "essential $\epsilon$-pseudospectra" can be found in \cite{Je}, and probably in many other places. Some of them are also interesting in a Banach space setting, but many of them are fitting our $C^*$-algebraic scheme. 
}
\end{rem}

For a non-unital $C^*$-algebra $\mathscr E$, its minimal unitization (obtained by adjoining a unit) is denoted by $\mathscr E^\bullet$. Whenever $\mathscr E$ is unital, we could set simply $\mathscr E^\bullet\!:=\mathscr E$. Concentrating on the non-unital case, if $\mu:\E\to\F$, we set $\mu^\bullet:\E^\bullet\to\F^\bullet$ by $\mu^\bullet(E+\lambda):=\mu(E)+\lambda$\,. The unital case allows a slightly simpler approach.

\begin{thm}\label{prune}
We are placed in the Setting \ref{orange}. Let $h\in{\sf C}^*(\Xi)^\bullet$ and $h_Y:={\sf R}_{XY}^\bullet(h)\in{\sf C}^*[\Xi(Y)]^\bullet$, where ${\sf R}_{XY}$ is a particular case of the restriction maps of the previous subsection. Then 
\begin{equation*}\label{dezvoltat}
\Si_{\rm ess}^{{\sf C}^*\![\Xi(X_0)]}\big(h\!\mid\!{\sf C}^*(\Xi)^\bullet\big)=\Si_{\rm ess}^{{\sf C}^*\![\Xi(Y_0]}\big(h_Y\!\mid\!{\sf C}^*[\Xi(Y)]^\bullet\big)\,.
\end{equation*}
\end{thm}

\begin{proof}
We label canonical injections and restrictions by indicating the unit spaces over which the corresponding groupoids are built. 
By Lemma \ref{matraguna}, one has the following commutative diagram (recall that ${\sf C}^*(\Xi)={\sf C}^*[\Xi(X)]$\,):
\begin{equation*}\label{coliflor}
\begin{diagram}
\node{0}\arrow{e,r}{}
\node{{\sf C}^*[\Xi(X_0)]}\arrow{e,t}{{\sf J}_{X_0X}} \arrow{s,l}{{\sf R}_{X_0Y_0}}\node{{\sf C}^*(\Xi)}\arrow{s,r}{{\sf R}_{XY}}\arrow{e,t}{{\sf R}_{XX_\infty}}\node{{\sf C}^*\big[\Xi(X_\infty)\big]}\arrow{s,r}{{\sf id}}\arrow{e,r}{}\node{0}\\ 
\node{0}\arrow{e,r}{}\node{{\sf C}^*[\Xi(Y_0)]} \arrow{e,t}{{\sf J}_{Y_0Y}} \node{{\sf C}^*[\Xi(Y)]}\arrow{e,t}{{\sf R}_{YX_\infty}}\node{{\sf C}^*\big[\Xi(X_\infty)\big]}\arrow{e,r}{}\node{0}
\end{diagram}
\end{equation*}
The horizontal arrows are $^*$-morphisms, while ${\sf R}_{X_0Y_0}$ and ${\sf R}_{XY}$ are only linear involutive contractions. The two horizontal short sequences are exact (Lemma \ref{matraguna} (ii)). A similar statement holds for the diagram involving the minimal unitalizations and the canonical extensions of maps
\begin{equation}\label{colifloc}
\begin{diagram}
\node{0}\!\arrow{e,r}{}\!
\node{{\sf C}^*[\Xi(X_0)]}\arrow{e,t}{{\sf J}_{X_0X}} \arrow{s,l}{{\sf R}_{X_0Y_0}}\node{{\sf C}^*(\Xi)^\bullet}\arrow{s,r}{{\sf R}^\bullet_{XY}}\arrow{e,t}{{\sf R}^\bullet_{XX_\infty}}\node{{\sf C}^*\big[\Xi(X_\infty)\big]^\bullet}\arrow{s,r}{{\sf id}^\bullet}\arrow{e,r}{}\node{0}\\ 
\node{0}\!\arrow{e,r}{}\node{{\sf C}^*[\Xi(Y_0)]} \arrow{e,t}{{\sf J}_{Y_0Y}} \node{{\sf C}^*[\Xi(Y)]^\bullet}\arrow{e,t}{{\sf R}^\bullet_{YX_\infty}}\node{{\sf C}^*\big[\Xi(X_\infty)\big]^\bullet}\arrow{e,r}{}\node{0}
\end{diagram}
\end{equation}
Thus we have canonical isomorphisms
\begin{equation}\label{simeza}
{\sf C}^*(\Xi)^\bullet/{\sf C}^*\big[\Xi(X_0)\big]\cong{\sf C}^*\big[\Xi(X_\infty)\big]^\bullet\cong{\sf C}^*\big[\Xi(Y)\big]^\bullet/{\sf C}^*\big[\Xi(Y_0)\big]\,.
\end{equation}
Then
$$
\begin{aligned}
\Si_{\rm ess}^{{\sf C}^*\![\Xi(X_0)]}\big(h\!\mid\!{\sf C}^*(\Xi)^\bullet\big)&=\Si\Big({\sf R}_{XX_\infty}^\bullet\!(h)\!\mid\!{\sf C}^*\big[\Xi(X_\infty)\big]^\bullet\Big)\\
&=\Si\Big({\sf R}^\bullet_{YX_\infty}\!\big[{\sf R}^\bullet_{XY}(h)\big]\!\mid\!{\sf C}^*\big[\Xi(X_\infty)\big]^\bullet\Big)\\
&=\Si\Big({\sf R}^\bullet_{YX_\infty}\!\big(h_Y\big)\!\mid\!{\sf C}^*\big[\Xi(X_\infty)\big]^\bullet\Big)\\
&=\Si_{\rm ess}^{{\sf C}^*\![\Xi(Y_0)]}\big(h_Y\!\mid\!{\sf C}^*[\Xi(Y)]^\bullet\big)\,.
\end{aligned}
$$
In the first and the last equalities we used the fact that $\Si$ is left invariant under the isomorphisms appearing in \eqref{simeza}. This allowed replacing images in abstract quotients (needed to define the two $\Si_{\rm ess}$-sets) by a common image in the concrete $C^*$-algebra ${\sf C}^*\big[\Xi(X_\infty)\big]^\bullet$. The second equality follows from the diagram, while the third one is just the definition of $h_Y$\,.
\end{proof}

\begin{ex}\label{xample}
{\rm Let us now assume that $\Xi$ is \textit{\'etale} (i.e. ${\rm r}:\Xi\to\Xi$ is a local homeomorphism) and that $X$ is compact.  Then $\Xi(Y)$ is also \'etale and the $C^*$-algebras ${\sf C}^*(\Xi)$ and ${\sf C}^*[\Xi(Y)]$ are unital (see the discussion in \cite[Sect.\,3.4]{Be} for more details). In addition, by \cite[Prop.\,II.4.1,\,II.4.2]{Re}, there are contractive linear embedings ${\sf C}^*(\Xi)\hookrightarrow C_0(\Xi)$ and ${\sf C}^*[\Xi(Y)]\hookrightarrow C_0[\Xi(Y)]$\,, so {\it all the elements of the reduced groupoid algebras can be seen as continuous functions, decaying at infinity (in the direction of the fibres)}. In this case, {\it the extended restriction operation ${{\sf R}_{XY}}$ is the concrete restriction of functions}.
}
\end{ex}

\section{The Decomposition Principe for bounded operators}\label{gramofoni}

 Still assuming Setting \ref{orange}, we are going to use now a special notation for a special case. If $\Si$ is a $C^*$-algebraic spectral set and $H\in\mathbb B(\H)$ is a bounded linear operator in a Hilbert space one sets
\begin{equation}\label{ofi}
\Si_{\rm ess}(H):=\Si\big(\gamma(H)\!\mid\!\mathbb B(\H)/\mathbb K(\H)\big)\,,
\end{equation}
where $\gamma:\mathbb B(\H)\to\mathbb B(\H)/\mathbb K(\H)$ is the canonical surjection to the Calkin algebra. For $\Si={\sf sp}$ one gets the essential spectrum. Other cases are essential numerical ranges and essential $\epsilon$-pseudospectra, for example.

\smallskip
To transform Theorem \ref{prune} into a significant result referring to operators, we look now for faithful representations of the groupoid $C^*$-algebras, sending the relevant ideals, isomorphically, to the ideal of compact operators. The following setting is convenient:

\begin{defn}\label{hermione}
The Hausdorff amenable locally compact $\si$-compact groupoid $\Xi$ is called {\it standard} if the unit space $\,X$ has a dense open orbit $X_0$ (sometimes called {\it the main orbit}), second countable for the induced topology, and the isotropy of the main orbit is trivial ($\,\Xi(a)=\{a\}\,$ for every $a\in X_0$)\,. We write $X_\infty:=X\setminus X_0$ and assume that it is compact and non-empty, $X$ being a compactification of $X_0$\,.
\end{defn}

The next two remarks are describing two important issues that occur in the standard case.

\begin{rem}\label{indoua}
{\rm The (invariant) reduction $\Xi(X_0)$ to the orbit $X_0$ is isomorphic to the pair groupoid $X_0\!\times\!X_0$\,. And then, by \cite[Th.\,3.1]{MRW}, one has 
$$
{\sf C}^*[\Xi(X_0)]\cong\mathbb K\big[L^2(X_0,\mu)\big]\otimes {\sf C}^*[\Xi(a)]
$$ 
for some full measure $\mu$ on $X_0$\,. In particular ${\sf C}^*[\Xi(X_0)]$ is an elementary $C^*$-algebra, specifically it is isomorphic to the $C^*$-algebra of compact operators in the Hilbert space $L^2(X_0,\mu)$\,.}
\end{rem}

\begin{rem}\label{harry}
{\rm In addition, for each $a\in X_0$\,, the restriction 
$$
{\rm r}_a:={\rm r}|_{\Xi_a}:\Xi_a=\Xi(X_0)_a\to X_0
$$ 
is surjective (since $X_0$ is an orbit) and injective (since the isotropy is trivial). One transports the measure $\lambda_a$ to a (full Radon) measure $\mu$ on $X$ (independent of  $a$\,, by the invariance of the Haar system) and gets a representation 
$$
\Pi_0:{\sf C}^*(\Xi)\to\mathbb B\big[L^2(X_0,\mu)\big]\,,
$$ 
called {\it the vector representation}. Let 
$$
\rho_a:L^2(X_0;\mu)\to L^2\big(\Xi_a;\lambda_a\big)\,,\quad \rho_a(u):=u\circ {\rm r}_a\,.
$$
Then $\Pi_0$ is defined to be unitarily equivalent to the regular representation $\Pi_a$\,:
\begin{equation*}\label{ibidem}
\Pi_0(f)=\rho_a^{-1}\Pi_a(f)\rho_a\,,\quad \Pi_0(f)u=\big[f\star\!(u\circ {\rm r}_a)\big]\circ {\rm r}_a^{-1}.
\end{equation*}}
\end{rem}

\begin{lem}\label{feisful}
We denote by $\Pi_0^\bullet$ the natural extension of $\,\Pi_0$ to the minimal unitization ${\sf C}^*(\Xi)^\bullet$. The representations $\Pi_0$ and $\Pi_0^\bullet$ are faithful.
\end{lem}

\begin{proof}
Since $X_0$ is an open dense orbit, the representation $\Pi_a$ is faithful, by \cite[Prop.\,2.7]{BB} (relying on \cite[Cor.\,2.4]{KS}). Thus the vector representation $\Pi_0$ is also faithful. If ${\sf C}^*(\Xi)$ is unital, the proof is already finished.

\smallskip
In the non-unital case, {\it it is enough to show that $\Pi_0(f)$ cannot be the identity operator for any $f\in{\sf C}^*(\Xi)$}\,. To explain this, note that injectivity fails if and only if  $\Pi^\bullet_0(h,\lambda)=\Pi_0(h)+\lambda{\rm id}=0$ may happen for $(h,\lambda)\ne(0,0)$\,. Clearly $\lambda\ne 0$\,, since $\Pi_0$ is already known to be injective. But then the non-injectivity of $\Pi_0^\bullet$ reads $\Pi_0(-h/\lambda)={\rm id}$ without $f:=-h/\lambda$ being null.

\smallskip
So let us check the stated assertion. If $\Pi_0(f)={\rm id}$\,, then for every $g\in{\sf C}^*(\Xi)$ one would have
$$
fg=\Pi_0^{-1}({\rm id})\Pi_0^{-1}\big[\Pi_0(g)\big]=\Pi_0^{-1}\big[{\rm id}\,\Pi_0(g)\big]=g
$$
and similarly $gf=g$\,, implying that ${\sf C}^*(\Xi)$ is actually unital. This shows that $\Pi_0(f)={\rm id}$ is impossible.
\end{proof}

\begin{rem}\label{nectarine}
{\rm {\it Assume now the Setting \ref{orange} with $\Xi$ standard and suppose, in addition, that the complement $X_0\setminus Y_0$ is relatively compact.} A direct application of the definitions shows that $Y_0\subset Y$ is a dense open orbit under $\Xi(Y)$ and that the isotropy over points of $Y_0$ is trivial. So Lemma \ref{feisful} can be applied to the standard groupoid $\Xi(Y)$\,, leading to a vector representation 
$$
\Pi_{Y,0}:{\sf C}^*[\Xi(Y)]\to\mathbb B\big[L^2\big(Y_0,\mu_{Y_0}\big)\big]\,,
$$ 
that is faithful and extends faithfully to the unitization ${\sf C}^*[\Xi(Y)]^\bullet$. Taking $a\in Y_0\subset X_0$\,, it is easily seen that $\mu_{Y_0}$\,, defined to be the image of $\lambda(Y)_a=\lambda_a|_{\Xi(Y)_a}$ through ${\rm r}_{Y,a}:={\rm r}|_{\Xi(Y_0)_a}$, is actually the restriction of $\mu$ to $Y_0$\,.}
\end{rem} 

Let us set, for simplicity, $\H(X_0):=L^2(X_0,\mu)$ and $\H(Y_0):=L^2\big(Y_0,\mu_{Y_0}\big)$\,. By Remark \ref{nectarine}, one can write 
$$
\H(X_0)=\H(Y_0)\oplus L^2(X_0\!\setminus\!Y_0)\,,
$$ 
leading to a canonical injection ${\sf j}_{XY}:\H(Y_0)\to\H(X_0)$ (extension by null values) and to a canonical projection ${\sf p}_{XY}:\H(X_0)\to\H(Y_0)$ (basically a restriction). We define the linear involutive contraction
\begin{equation*}\label{lichie}
\mathfrak R_{XY}:\mathbb B\big[\H(X_0)\big]\to\mathbb B\big[\H(Y_0)\big]\,,\quad \mathfrak R_{XY}(T):={\sf p}_{XY}\circ T\circ{\sf j}_{XY}\,,
\end{equation*}
that clearly restricts to a linear involutive contraction 
$$
\mathfrak R^\mathbb K_{XY}:\mathbb K\big[\H(X_0)\big]\to\mathbb K\big[\H(Y_0)\big]\,.
$$ 
{\it In general, it is not a $C^*$-morphism.} However, $\mathfrak R_{XY}$ and $\mathfrak R^\mathbb K_{XY}$ are conditional expectations.

\begin{thm}\label{aguacate}
Let $\,\Xi$ be a standard groupoid over $X$ with main orbit $X_0$ (cf. Definition \ref{hermione}) and $Y$ a tame subset  of $X$ containing $X_\infty\!:=X\!\setminus\! X_0$\,, such that $Y_0:=Y\cap X_0$ has a relatively compact complement in $X_0$\,. Let $h\in{\sf C}^*(\Xi)^\bullet$ and set
$$
H=\Pi_0^\bullet(h)\in \Pi_0^\bullet\big[{\sf C}^*(\Xi)^\bullet\big]\subset\mathbb B\big[\H(X_0)\big]
$$
and
$$
H_Y:=\mathfrak R_{XY}(H)\in\mathbb B\big[\H(Y_0)\big]\,.
$$ 
For every $C^*$-algebraic spectral set $\Si$\,, using the notation \eqref{ofi}, one has $\Si_{\rm ess}(H)=\Si_{\rm ess}(H_Y)$\,.
\end{thm}

\begin{proof}
As it will be shown in the last paragraph of the proof, the main concrete step is to show that the next diagram commutes:
\begin{equation}\label{conopida}
\begin{diagram}
\node{\mathbb K[\H(X_0)]}\arrow{e,t}{}\arrow{s,t}{}\node{}\arrow{e,t}{}\node{}\arrow{e,t}{}\node{\mathbb B[\H(X_0)]}\arrow{s,r}{}\\
\node{}\arrow{s,t}{\mathfrak R^\mathbb K_{XY}}\node{{\sf C}^*[\Xi(X_0)]}\arrow{nw,t}{\Pi^\mathbb K_0}\arrow{e,t}{{\sf J}_{X\X}} \arrow{s,l}{{\sf R}_{X_0Y_0}}\node{{\sf C}^*(\Xi)^\bullet}\arrow{s,r}{{\sf R}_{XY}^\bullet}\arrow{ne,t}{\Pi_0^\bullet}\node{}\arrow{s,r}{\mathfrak R_{XY}}\\ 
\node{}\arrow{s,t}{}\node{{\sf C}^*[\Xi(Y_0)]}\arrow{sw,t}{\Pi_{Y,0}^\mathbb K}\arrow{e,t}{{\sf J}_{Y\Y}} \node{{\sf C}^*[\Xi(Y)]^\bullet}\arrow{se,t}{\Pi_{Y,0}^\bullet}\node{}\arrow{s,r}{}\\
\node{\mathbb K[\H(Y_0)]}\arrow{e,t}{}\node{}\arrow{e,t}{}\node{}\arrow{e,t}{}\node{\mathbb B[\H(Y_0)]}
\end{diagram}
\end{equation}
Notations as those of Lemma \ref{matraguna} were used in this diagram. 
We already know from \eqref{colifloc} that the central part of the diagram \eqref{conopida} commutes. It is convenient to use self-explaining "geographical terminology" for the other subdiagrams in \eqref{conopida}. 

\smallskip
The North and South diagrams also commute. In addition $\Pi^\bullet_0$ and $\Pi^\bullet_{Y,0}$ are faithful, by Lemma \ref{feisful}, while the restrictions $\Pi_{0}^\mathbb K$ and $\Pi_{Y,0}^\mathbb K$ are isomorphisms. We summarize useful information as the isomorphism relations
\begin{equation}\label{summetinf}
{\sf C}^*(\Xi)^\bullet/{\sf C}^*[\Xi(X_0)]\cong\Pi_0^\bullet\big[{\sf C}^*(\Xi)^\bullet\big]/\mathbb K[\H(X_0)]\subset\mathbb B[\H(X_0)]/\mathbb K[\H(X_0)]\,,
\end{equation}
\begin{equation}\label{sumetinf}
{\sf C}^*[\Xi(Y)]^\bullet/{\sf C}^*[\Xi(Y_0)]\cong\Pi_{Y,0}^\bullet\big[{\sf C}^*[\Xi(Y)]^\bullet\big]/\mathbb K[\H(Y_0)]\subset\mathbb B[\H(Y_0)]/\mathbb K[\H(Y_0)]\,.
\end{equation}

If the East diagram commutes, the West diagram will also commute, since it is just a well-defined restriction of the former (actually, the commutativity of the West diagram will not be used below). To check the commutativity of the  East diagram, it is enough to show that
\begin{equation}\label{citron}
{\sf p}_{XY}\circ\Pi_0(f)\circ{\sf j}_{XY}=\Pi_{Y,0}\big(f|_{\Xi(Y)}\big)\,,\quad\forall\,f\in C_{\rm c}(\Xi)\,;
\end{equation}
then the extensions and the unitizations are taken into account easily. To check \eqref{citron}, by the definition of the representations $\Pi_0$ and $\Pi_{Y,0}$\,, it is enough to show that (as functions on $Y_0$)
\begin{equation}\label{grapefruit}
{\sf p}_{XY}\big[\big(f\star(w\circ {\rm r}_a)\big)\circ {\rm r}_a^{-1}\big]=\big[f|_{\Xi(Y)}\star_Y(w\circ {\rm r}_{Y\!,a})\big]\circ {\rm r}_{Y\!,a}^{-1}
\end{equation}
for every $f\in C_{\rm c}(\Xi)$ and $w\in C_{\rm c}(Y_0)$\,, where $a\in Y_0$\,, $\,\star$ is the convolution with respect to the initial groupoid $\Xi$ and $\star_Y$ is the convolution with respect to the restricted groupoid $\Xi(Y)$\,. In the left hand side, $w$ has been tacitly extended by null values outside $Y_0$\,. But \eqref{grapefruit} is essentially \eqref{passion}, with modified notations. Note that $\tilde v:=w\circ {\rm r}_a$ is supported on $\Xi(Y)_a$\,, so it coincides with $w\circ {\rm r}_{Y\!,a}$\,, while ${\sf p}_{XY}$ is just the operation of restricting to $Y_0$\,.

\smallskip
Now we show equality of the essential $\Si$-spectra, using the definitions and the properties of $\Si$\,. On one hand, by \eqref{summetinf}, one has
\begin{equation*}\label{sfecla}
\Si_{\rm ess}(H)=\Si_{\rm ess}[\Pi_0^\bullet(h)]=\Si_{\rm ess}^{{\sf C}^*[\Xi(X_0)]}\big(h\!\mid\!{\sf C}^*(\Xi)^\bullet\big)\,.
\end{equation*}
On the other hand, by \eqref{sumetinf} and the key fact that the East diagram commutes (see \eqref{citron}), one also has
$$
\begin{aligned}
\Si_{\rm ess}(H_Y)&=\Si_{\rm ess}\big[\mathfrak R_{XY}(\Pi_0^\bullet(h))\big]\\
&=\Si_{\rm ess}\big[\Pi_{Y,0}^\bullet\big({\sf R}_{XY}^\bullet(h)\big)\big]\\
&=\Si_{\rm ess}^{{\sf C}^*[\Xi(Y_0)]}\big(h_Y\!\mid\!{\sf C}^*[\Xi(Y)]^\bullet\big)\,,
\end{aligned}
$$
where $h_Y:={\sf R}_{XY}^\bullet(h)$\,. The two expressions are equal, by Theorem \ref{prune}.
\end{proof}

\begin{cor}\label{desprespectru}
In the framework of Theorem \ref{aguacate}, $H$ and $H_Y$ have the same essential spectrum. The operator  $H_Y$ is Fredholm if and only if $H$ is Fredholm. 
\end{cor}

\begin{proof}
The spectrum $\Si={\sf sp}$ is the most important example of $C^*$-algebraic spectral set, to which Theorem \ref{aguacate} apply.

\smallskip
The assertion about the Fredholm property follows immediately from Atkinson's Theorem, saying that an operator acting in a Hilbert space is Fredholm if and only if its image in the Calkin algebra is invertible, i.\,e. if $0$ does not belong to its essential spectrum.
\end{proof}

\section{The Decomposition Principle for partial actions}\label{gerfomoni}

Let us now assume that $X_0$ is a unimodular, amenable, second countable Hausdorff locally compact group with unit ${\sf e}$ and fixed Haar measure $\m$ and that the action of $X_0$ on itself by left translations extends to a continuous action of $X_0$ on an arbitrary $\si$-compact compactification $X\!:=X_0\sqcup X_\infty$\,. So we have a continuous map 
\begin{equation*}\label{ficus}
X_0\times X\ni(a,z)\to\th(a,z)\equiv\th_a(z)\in X
\end{equation*}
such that each $\th_a$ is a homeomorphism, $\th_a(z)=az$ if $z\in X_0$ and $\th_a\circ\th_b=\th_{ab}$ for every $a,b\in X_0$\,. We require the group $X_0$ to sit as an open dense orbit in $X$ and the closed boundary $X_\infty$ is clearly invariant. 

\smallskip
To such a dynamical system $(X,\th,X_0)$ one associates {\it the transformation groupoid} $\Xi:=X\!\rtimes_\th\!X_0$\,. It is the topological product $X\!\times\!X_0$\,, with groupoid maps \eqref{fitze}. From our assumptions, it follows that {\it the transformation groupoid is standard, with main orbit $X_0$}\,. For every $z\in X$, setting $\lambda_z\!:=\delta_z\!\otimes\!\m$\,, one gets a Haar system of the groupoid.

\smallskip
Let $Y_0$ be a closed subset  of $X_0$ with relatively compact complement; we set $Y\!:=Y_0\sqcup X_\infty$\,. It can be shown that if, in addition, $Y_0$ is the closure of its interior, it will be a tame set. Relying on a result of Nica \cite[Sect.1]{Ni},  this has been shown for the case of connected simply connected nilpotent groups in \cite[Prop.4.7]{MN1}. See also \cite{RS}. But the proof holds, with the same proof, in our  more general case. Thus we are placed in the Setting \ref{orange}, we construct the restricted groupoid $\Xi(Y)$ and both $\Xi$ and $\Xi(Y)$ are standard. One can present $\Xi(Y)$ as the groupoid associated to a partial dynamical system \cite{Ex2}; this is not important here.

\smallskip
Thus the results of the previous sections apply. {\it Theorem \ref{prune} may be immediately particularized to this case.} To write down Theorem \ref{aguacate} in the framework of this section, one still has to calculate 
$$
H\!:=\Pi_0(h)\quad{\rm and}\quad  H_Y\!:=\mathfrak R_{XY}(H)={\sf p}_{XY}\circ H\circ{\sf j}_{XY}\,.
$$ 
By the proof of Theorem \ref{aguacate}  one also has $H_Y=\Pi_{Y,0}\big[{\sf R}_{XY}(h)\big]$\,, and this yields another way to compute. We take $h\in C_{\rm c}(\Xi)$ for simplicity, but more general options are possible (as $h\in L^1\big(X_0;C(X)\big)$ for instance).  Being continuous, $h$ is uniquely determined by the values taken on $X_0\!\times\!X_0$\,, which is dense in $\Xi=X\times\!X_0$\,. Whenever the group $X_0$ is discrete, the groupoid $\Xi$ is \'etale and we are placed in the setting of Example \ref{xample}, so actually restrictions act concretely on the entire (reduced) groupoid algebra. 

\smallskip
 Anyhow, one arrives quickly to the formulae
\begin{equation*}\label{vectorial}
(Hu)(a)=\int_{X_0}\! h\big(c,ac^{-1}\big)u(c)d\m(c)\,,\quad u\in L^2(X_0;\m)\,,\quad a\in X_0\,,
\end{equation*}
\begin{equation*}\label{vecdorial}
(H_Y v)(b)=\int_{Y_0}\! h\big(c,bc^{-1}\big)v(c)d\m(c)\,,\quad v\in L^2(Y_0;\m)\,,\quad b\in Y_0\,.
\end{equation*}
Let us consider the particular case 
$$
h(x,a)\!:=\varphi(x)\psi(a)\,,
$$ 
where $\varphi\in C(X)$ (identified with a continuous function on $X_0$ that can be extended continuously on $X$) and (for simplicity) $\psi\in C_{\rm c}(X_0)$\,. Then
\begin{equation*}\label{vectorialy}
(Hu)(a)=\int_{X_0}\!\psi\big(ac^{-1}\big)(\varphi u)(c)d\m(c)\,,
\end{equation*}
meaning that one may write $H={\rm Conv}(\psi){\rm Mult}(\varphi)$\,, the product between a multiplication and a convolution operator. Then one also has 
$$
H_Y\!={\rm Toep}_{Y_0}\!(\psi){\rm Mult}(\varphi|_{Y_0})\,,
$$ 
which is the product between a multiplication operator in $L^2(Y_0)$ and {\it  a Toeplitz-like operator}, i.\,e.\ the compression of a convolution operator. In this case, or in general, {\it the operators $H$ and $H_Y$ have the same $\Si_{\rm ess}$-set, in particular the same essential spectrum, the same essential numerical range and the same $\epsilon$-essential pseudospectrum}. 
In addition, {\it ${\rm id}_{L^2(X_0)}\!+\!H$ is Fredholm in $L^2(X_0)$ if and only if $\,{\rm id}_{L^2(Y_0)}\!+\!H_Y$ is Fredholm in $L^2(Y_0)$}\,.

\begin{cor}\label{pontryagin}
Let $X_0$ be a second countable Abelian locally compact non compact group with Pontryagin dual $\widehat X_0$ and let $\psi\in C_c(X_0)$ (taking $\psi\in L^1(X_0)$ is also possible). For every closed set $Y_0\subset X_0$ that is the closure of its interior and has relatively compact complement, one has
\begin{equation}\label{frotac}
{\sf sp}_{\rm ess}\big[{\rm Toep}_{Y_0}\!(\psi)\big]=\overline{\Big\{\int_{X_0}\!\chi(a)\psi(a)d\m(a)\,\big\vert \,\chi\in\widehat{X}_0\Big\}}\,.
\end{equation}
\end{cor}

\begin{proof}
One can take here $X\!:=X_0\sqcup\{\infty\}$\,, the Alexandrov compactification of $X_0$\,, the point at infinity being a fixed point for the action. With notations as above, we work with $H={\rm Conv}(\psi)$ and $H_Y\!={\rm Toep}_{Y_0}(\psi)$ ($\varphi(\cdot)=1$ is a legitimate choice). 

\smallskip
By our results, we know that these two operators have the same essential spectrum. So one only needs to show that the essential spectrum of ${\rm Conv}(\psi)$ coincides with the r.\,h.\,s.\! of \eqref{frotac}, which  is simply the closure of the range $\widehat\psi\big(\widehat X_0\big)$ of the Fourier transform of the function $\psi$\,. The convolution operator ${\rm Conv}(\psi)$ is unitarily equivalent (via the Fourier transformation) with the operator of multiplication by $\widehat\psi$ in $L^2\big(\widehat X_0\big)$\,, so its (full) spectrum is $\overline{\widehat\psi\big(\widehat X_0\big)}$\,. 

\smallskip
To finish the proof, one still has to show that the spectrum of ${\rm Conv}(\psi)$ is purely essential. This follows in a standard way from the fact that the group is not compact, since the convolution operators commute with translations, and translations by $b$ converge weakly to zero when $b\to\infty$\,.
\end{proof}

Corollary \ref{pontryagin} may be interpreted as expressing the essential spectrum of ${\rm Toep}_{Y_0}(\psi)$ as the total spectrum of a 'limit operator' corresponding to the point at infinity, that may be defined via the amenable groupoid $(X\!\rtimes_\th\!X_0)|_{Y_0\sqcup\{\infty\}}$ (see references \cite{AZ,CNQ, Co} for example).

\section{The Decomposition Principle for discrete metric spaces}\label{germanofoni}

We are going to fix a {\it uniformly discrete metric space $(X_0,d)$ with bounded geometry, satisfying Yu's Property A}. To avoid trivialities, we assume it to be infinite. {\it Uniform discreteness} means that, for some $\delta>0$\,, one has $d(x,x')\ge\delta$ if $x\ne x'$\,. {\it Bounded geometry} means that for every positive $r$, all the balls of radio $r$ have less then $N_r<\infty$ elements. We do not make explicit {\it Property A}, referring to \cite[Ch.5]{BOz}; for the particular case of countable discrete groups (with any invariant metric), it is equivalent to exactness of the group. Its role here is to insure amenability of the groupoid introduced below.

\smallskip
 Linear bounded operators $H$ in the Hilbert space $\ell^2(X_0)$ are given by (suitable) "$X_0\!\times\!X_0$-matrices" with complex entries:
\begin{equation}\label{fraier}
[H(u)](a)=\sum_{a'\in X_0}\!H(a,a')u(a')\,.
\end{equation}
Sometimes we identify the operator with the corresponding "matrix".

\begin{defn}\label{frayer}
\begin{enumerate}
\item[(i)]
We say that $H\in\mathbb B\big[\ell^2(X_0)\big]$ is {\it a band operator} if there is some positive $r$ such that $H(a,a')=0$ if $d(a,a')>r$ and if $\sup_{a,a'}|H(a,a')|<\infty$\,.
\item[(ii)]
A linear bounded operator is {\it band dominated} if it is a norm-operator limit of band operators.
\end{enumerate}
\end{defn}

It is known that the band operators form a $^*$-subalgebra $\mathbb C(X,d)$ and the band-dominated operators form a unital $C^*$-subalgebra $C^*(X_0,d)$ of $\mathbb B\big[\ell^2(X_0)\big]$\,. The later one is called {\it the uniform Roe algebra of the metric space}. 

\smallskip
Let now $Y_0$ be a subset of $X_0$ with finite complement. Assuming that \eqref{fraier} is a band-dominated operator in $\ell^2(X_0)$\,, it is clear that
\begin{equation}\label{freier}
\big[H_{Y_0}(v)\big](b)=\sum_{b'\in Y_0}\!H(b,b')v(b')\,,\quad\forall\,b\in Y_0
\end{equation}
defines a band-dominated operator in $\ell^2(Y_0)$\,.

\begin{thm}\label{uvert}
Let $\Si$ be a $C^*$-algebraic spectral set. Then $\Si_{\rm ess}(H)=\Si_{\rm ess}\big(H_{Y_0}\big)$\,.
\end{thm}

To deduce the present result from Theorem \ref{aguacate}, we need to recall the groupoid model behind the uniform Roe algebras, cf \cite{STY,Roe}.

\smallskip
Let us denote by $\beta X_0\equiv X$ {\it the Stone-\u Cech compactification} of the discrete space $X_0$ and by $X_\infty\!:=\beta X_0\!\setminus\!X_0$ the corresponding boundary (corona set). Then $X_0$ is a dense open subset of $\beta X_0$\,. For every $r\ge 0$ one sets 
$$
\Delta_r\!:=\{(a,a')\in X_0\!\times\! X_0\!\mid d(a,a')\le r\}\,,
$$ 
with closure $\overline{\Delta_r}$ in $\beta(X_0\!\times\!X_0)$\,, and then 
\begin{equation*}\label{abiert}
\Xi(X_0,d)\!:=\bigcup_{r>0}\,\overline{\Delta_r}\,.
\end{equation*}
It is explained in \cite[Sect.10.3]{Roe} that, although $\beta(X_0\!\times\!X_0)$ and $\beta X_0\!\times\!\beta X_0$ are not homeomorphic, the sets $\overline{\Delta_r}$ may be regarded as subsets of $\beta X_0\!\times\!\beta X_0$\,. Finally, $\Xi(X_0,d)$ is an open subset of $\beta X_0\!\times\!\beta X_0$\,, and the pair groupoid structure restricts to a groupoid structure on $\Xi(X_0,d)$\,. One defines the topology on $\Xi(X_0,d)$\,: the subset $U\subset \Xi(X_0,d)$ is declared open if and only if $U\cap\overline{\Delta_r}$ is open in $\overline{\Delta_r}$ for any $r>0$ (this is {\it not} the subspace topology).

\begin{prop}\label{delaroe}
\begin{enumerate}
\item[(i)]
$\,\Xi(X_0,d)$ is an amenable $\si$-compact \'etale standard groupoid over the unit space $\beta X_0$\,, with main orbit $X_0$\,.
\item[(ii)]
Its groupoid $C^*$-algebra ${\sf C}^*\!\big[\Xi(X_0,d)\big]$ is isomorphic to the uniform Roe $C^*$-algebra $C^*\!(X_0,d)$\,. This isomorphism sends the ideal ${\sf C}^*\!(X_0\!\times\!X_0)$\,, corresponding to the main orbit, onto the ideal of all compact operators in $\ell^2(X_0)$\,.
\end{enumerate}
\end{prop}

\begin{proof}
This may be found in \cite[Sect.10.3,10.4]{Roe}; see also \cite[Ch.3]{STY} and \cite[App.C]{SWi}. Amenability is due to Property A. Let us only recall the action of the isomorphism $\Gamma:{\sf C}^*\!\big[\Xi(X_0,d)\big]\to C^*\!(X_0,d)$ on continuous compactly supported functions. Such a function $h$ is actually supported on some set $\overline{\Delta_r}$\,, so it may be seen as a continuous function $:\Delta_r\!\to\mathbb C$\,. Then $\Gamma$ transforms it into the operator \eqref{freier} with $H=h$\,. 
\end{proof}

\begin{rem}\label{seapropie}
It is easy to see that if $a\in X_0\subset\beta X_0=\Xi(X_0,d)^{(0)}$, the ${\rm d}$-fiber in $a$ is $\Xi(X_0,d)_a\!=X_0\!\times\!\{a\}$\,. The restriction of the ${\rm r}$-map to this fiber, needed for the vector representation in Remark \ref{harry}, just sends $(a',a)$ into $a'\in X_0$\,. It turns out that the isomorphism $\Gamma$ may be interpreted as the vector representation. This follows from Remark \ref{harry} and from the fact that, in our case, groupoid convolution is the usual composition of kernels (matrices).
\end{rem}

Assuming, as above, that $X_0\!\setminus\!Y_0$ is finite, we set $Y\!:=Y_0\sqcup X_\infty$\,. Since we are in an \'etale setting, the chosen Haar measures are counting measures and tameness is not a problem; we are placed in the Setting \ref{orange}. It is clear that \eqref{freier} is obtained from \eqref{fraier} by applying the procedure indicated before Theorem \ref{aguacate}.  {\it So Theorem \ref{uvert}. follows from Theorem \ref{aguacate}.}

\begin{rem}\label{narco}
Following the suggestions of a referee, to whom I am grateful, I will sketch the proof of a result that improves Theorem \ref{uvert}, based on the fact that, in the present case, $X_0\!\setminus\! Y_0$ is a finite set, with the consequence that $\ell^2(Y_0)$ is a finite-dimensional Hilbert space (this does not happen in other parts of the article). So, more generally, let $\H=\mathcal L\oplus\mathcal L^\perp$ be an orthogonal decomposition with ${\rm dim}\big(\mathcal L^\perp\big)<\infty$\,. Correspondingly ${\rm id}_\H={\sf p}+{\sf p}^\perp$, where we also regard the two orthogonal projections as surjections on their natural ranges, respectively. By ${\sf j}={\sf p}^*\!:\mathcal L\to\H$ we denote the canonical injection.  For $H\in\mathbb B(\H)$ we set 
$$
L\equiv\mathfrak R(H):={\sf p}H{\sf j}\in\mathbb B(\mathcal L)\,.
$$
The operators $H$ and $L$ act in different Hilbert spaces; we cannot say sharply that 'their difference is a compact operator'. The mapping $\mathfrak R:\mathbb B(\H)\to\mathbb B(\mathcal L)$ is not multiplicative, but it  preserves compactness and in addition {\it it becomes a $C^*$-morphism of the Calkin algebras.} Let us set
$$
\mu:\mathbb B(\mathcal H)\to(\mathbb B/\mathbb K)(\mathcal H)\,,\quad\nu:\mathbb B(\mathcal L)\to(\mathbb B/\mathbb K)(\mathcal L)
$$
for the quotient maps. Defining (correctly) 
$$
\mathscr R:(\mathbb B/\mathbb K)(\mathcal H)\to(\mathbb B/\mathbb K)(\mathcal L)\quad{\rm by}\quad\mathscr R\circ\mu=\nu\circ\mathfrak R
$$ 
(an obvious diagram commutes), we see that it is actually a $C^*$-morphism. Multiplicativity follows easily  from the definitions, from the fact that ${\sf p}^\perp$ is finite rank and from
\begin{equation}\label{acid}
{\sf p}ST{\sf j}-{\sf p}S{\sf j}{\sf p}T{\sf j}={\sf p}S{\sf p}^\perp T{\sf j}\in\mathbb K(\mathcal H_Y)\,,\quad\forall\,S,T\in\mathbb B(\mathcal H)\,.
\end{equation}
In addition $\mathscr R$ is injective, since ${\sf p}H{\sf j}$ is compact in $\mathcal L$ if and only if $H$ is compact in $\H$\,.
{\it Now we show that $\Sigma_{\rm ess}(H)=\Sigma_{\rm ess}(L)$}\,, which will cover Theorem \ref{uvert} after some obvious particularizations. We know that $\Sigma_{\rm ess}$ is obtained from $\Sigma$ applied to the image in a quotient (a Calkin algebra in this case) and that $\Sigma$ is preserved under $C^*$-monomorphisms, which is used in the second of the following equalities. One writes
\begin{equation}\label{heroin}
\begin{aligned}
\Sigma_{\rm ess}(H)&=\Sigma\big(\mu(H)\,\big\vert\,(\mathbb B/\mathbb K)(\mathcal H)\big)=\Sigma\big(\mathscr R[\mu(H)]\,\big\vert\,(\mathbb B/\mathbb K)(\mathcal L)\big)\\
&=\Sigma\big(\nu[\mathfrak R(H)]\,\big\vert\,(\mathbb B/\mathbb K)(\mathcal L)\big)=\Sigma\big(\nu(L)\,\big\vert\,(\mathbb B/\mathbb K)(\mathcal L)\big)=\Sigma_{\rm ess}(L)
\end{aligned}
\end{equation}
and we are done.
\end{rem}

\begin{rem}\label{cocaine}
Such an argument is not possible if $\mathcal L^\perp$ has infinite dimension, as in most cases in this paper or in other situations. Then one really has to reduce the investigation to elements of suitable $C^*$-algebras $\mathfrak B$ of $\mathbb B(\H)$ and to suitable 'restriction maps' $\mathfrak R:\mathfrak B\to\mathfrak C\subset\mathbb B(\mathcal L)$\,. Under favorable circumstances, the induced map $\mathscr R:\mathfrak B/\mathbb K(\H)\to\mathfrak C/\mathbb K(\mathcal L)$ would become multiplicative and the proof survives. The groupoid techniques used in the second half of this article put into evidence such a case.
\end{rem}

\begin{rem}\label{georgescu}
In a very interesting article \cite{Ge}, V. Georgescu treats band dominated operators associated to non-discrete metric spaces. In search of a Decomposition Principle, this would not lead to finite dimensional corestrictions and would be technically more demanding. Finding a generalization of the Roe groupoid for this 'continuous case' would have both an intrinsic interest and usefulness for spectral applications.
\end{rem}

\subsection*{Acknowledgment} The author has been supported by the Fondecyt Project 1160359. He is grateful to Professors Victor Nistor and Jiawen Zhang for several useful discussions and to Professor Serge Richard for a critical reading of the manuscript. He is also grateful to anonymous referees for many useful remarks, which contributed to an improvement of the presentation.


\bigskip
\bigskip
DEPARTAMENTO DE MATEM\'ATICAS, 

\smallskip
FACULTAD DE CIENCIAS,

\smallskip
UNIVERSIDAD DE CHILE, 

\smallskip
SANTIAGO, CHILE.

\medskip
{\it E-mail address:} mantoiu@uchile.cl


\begin{thebibliography}{1}

\bibitem{Ab} F. Abadie: \textit{On Partial Actions and Groupoids}, Proc. Amer. Math. Soc. \textbf{132}(4), 1037--1047, (2004).

\bibitem{ALN} B. Ammann, R. Lauter and V. Nistor: \textit{Pseudodifferential Operators on Manifolds with a Lie Structure at Infinity}, Ann. of Math. (2), {\bf 165}(3), 717--747, ( 2007).

\bibitem{ADR} C. Anantharaman-Delaroche and J. Renault: \emph{Amenable Groupoids},  volume 36 of \emph{Monographies de l'Enseignement Math\'ematique}, Ens. Math. Monogr. Geneve, 2000.

\bibitem{Au} J. Auslander: \textit{Minimal Flows and Their Extensions}, North Holland, Amsterdam, 1988.

\bibitem{AZ} K. Austin and J. Zhang: \emph{Limit Operator Theory for Groupoids}, Trans. of the AMS, \textbf{373}(4), 2861--2911, (2020).

\bibitem{Be} S. Beckus: \textit{Spectral Approximation for Aperiodic Schr\"odinger Operators}, Friedrich-Schiller-Universit\"at, Jena, PhD Thesis.

\bibitem{BLLS1} S. Beckus, D. Lenz, M. Lindner and C. Seifert: \textit{On the Spectrum of Operator Families on Discrete Groups over Minimal Dynamical Systems}, Math. Z, \textbf{287}, 993--1007, (2017).

\bibitem{BLLS2} S. Beckus, D. Lenz, M. Lindner and C. Seifert: \textit{Note on Spectra of Non- Selfadjoint Operators Over Dynamical Systems}, Proc. of the Edinburgh Math. Soc. \textbf{61}(2), 371--386, (2018).

\bibitem{BBdN} S. Beckus, J. Bellissard and G. de Nittis: \textit{Spectral Continuity for Aperiodic Quantum Systems}, J. Funct. Anal. \textbf{275}(11), 2917--2977, (2018). 

\bibitem{BHZ} J. Bellissard, D. Herrmann and M. Zarrouati: \textit{Hull of Aperiodic Solids and Gap Labelling Theorems}, in \emph{Directions in Mathematical Quasicristals}, CIRM Monograph Series \textbf{13}, 207--259, 2000.

\bibitem{BIST} J. Bellissard, B. Iochum, E. Scoppola, and D. Testard: \textit{Spectral Properties of One-Dimensional Quasi-Crystals}, Comm. Math. Phys. \textbf{125}(3), 527--543, (1989).

\bibitem{BLM} F. Belmonte, M. Lein and M. M\u antoiu: \textit{Magnetic Twisted Actions on General Abelian $C^*$-Algebras}, J. Operator Th., \textbf{69}, 33--58, (2013).

\bibitem{BB} I. Belti\c t\u a and D. Belti\c t\u a: \textit{$C^*$-Dynamical Systems of Solvable Lie Groups}, Transformation Groups, \textbf{23}(3), 589--629, (2018).

\bibitem{BO} S. K. Berberian and G. H. Orland: \textit{On the Closure of the Numerical Range of an Operator}, Proc. Amer. Math. Soc, \textbf{18}, 499--503, (1967).

\bibitem{BD} F.\,F. Bonsall and J. Duncan: \textit{Numerical Ranges of Operators on Normed Spaces and of Elements of Operator Algebras}, Cambridge Univ. Press, 1971.

\bibitem{BD2} F.\,F. Bonsall and J. Duncan: \textit{Numerical Ranges II}, Cambridge Univ. Press, 1973.

\bibitem{BaH} J. H. Brown and A. an Huef: \textit{Decomposing the $C^*$-Algebras of Groupoid Extensions}, Proc. Amer. Math. Soc. {\bf 142}(4), 1261--1274, (2014).

\bibitem{BOz} N. Brown and N. Ozawa: \emph{$C^*$-Algebras and Finite-Dimensional Approximations}, \textbf{88}, Graduate Studies in Mathematics. American Mathematical Society, Providence, RI, 2008.

\bibitem{BM} H. Bustos and M. M\u antoiu: \textit{Twisted Peudo-Differential Operators on Type I Locally Compact Groups}, Illinois J. Math. \textbf{60}(2), 365--390, (2016).

\bibitem {CNQ} C. Carvalho, V. Nistor and Y. Qiao: \textit{Fredholm Conditions on Non-Compact Manifolds: Theory and Examples}, In Operator Theory, Operator Algebras, and Matrix Theory, volume 267 of Oper. Theory Adv. Appl., pages 79?122. Birkhäuser/Springer, Cham, 2018.

\bibitem{Co} R. C\^ome: \textit{The Fredholm Property for Groupoids is a Local Property}, Results in Mathematics, \textbf{74}(4), 160, (2019).

\bibitem{Da} E. B. Davies: \textit{Linear Operators and Their Spectra}, Cambridge University Press, 2007.

\bibitem{DS} C. Debord and G, Skandalis: \textit{Adiabatic Groupoid, Crossed Product by $\R_+^*$ and Pseudodifferential Calculus}, Adv. in Math. \textbf{257}, 66--91, (2013).

\bibitem{Di} J. Dixmier: \textit{Les $C^*$-alg\`ebres et leur repr\'esentations}, Cahiers Scientifiques, \textbf{XXIX}, Gauthier-Villars, Paris, 1969.

\bibitem{EE} D.\,B. Ellis and R. Ellis: \textit{Automorphisms and Equivalence Relations in Topological Dynamics}, London Mathematical Society Lecture Notes Series, \textbf{412}, 2013.

\bibitem{EH} E.\,G.\,Effros and F. Hahn: \textit{Locally Compact Transformation Groups and $C^*$-Algebras}, Memoirs of the A.M.S. \textbf{75}, (1975).

\bibitem{Ex2} R. Exel: \textit{Partial Dynamical Systems, Fell Bundles and Applications}, Providence, Rhode Island : American Mathematical Society, 2017.

\bibitem{FS} T. Fack and G. Skandalis: \textit{Sur les repr\'esentations et id\'eaux de la $C^*$-algebra d'un feuilletage}, J. Oper. Theory, \textbf{8}, 95--129, (1982).

\bibitem{Ge} V. Georgescu: \textit{On the Structure of the Essential Spectrum of Elliptic Operators on Metric Spaces}, J. Funct. Anal. \textbf{260}, 1734--1765, (2011).

\bibitem{Gi} E. Gillaspy: \textit{$K$-Theory and Homotopies of \,$2$-Cocycles on Transformation Groups}, J. Oper. Theory, \textbf{73}, 465--490, (2015).

\bibitem{GR} K.\,E. Gustafson and D.\,K.\,M. Rao:  \textit{Numerical Range}, New York, Springer, 1997.

\bibitem{HRS} R. Hagen, S. Roch and B. Silbermann: \textit{$C^*$-Algebras and Numerical Analysis}, Marcel Dekker Inc. New York, Basel, 2001.

\bibitem{HS} R. Heggert and C. Seifert: \emph{Limit Operators Techniques on General Metric Measure Spaces of Bounded Geometry},  Preprint ArXiV, 1908.01985v1, (2019).

\bibitem{HR} E. Hewitt and K. A. Ross: \textit{Abstract Harmonic Analysis}, Vol. I\,, Grundlehren der Mathematischen Wissenchaften (Springer, Berlin), 1979.

\bibitem{IMP} V. Iftimie, M. M\u antoiu and R. Purice: \textit{Magnetic Pseudodifferential Operators}, Publ. RIMS. \textbf{43}, 585-Ð623, (2007).

\bibitem{Je} A. Jeribi: \textit{Linear Operators and Their Essential Pseudospectra}, Apple Academic Press, 2018.


\bibitem{Kn} A.\,W. Knapp: \textit{Decomposition Theorem for Bounded Uniformly Continuous Functions on a Group}, Amer. J. Math. \textbf{88}, 901--914,  (1966).

\bibitem{KS} M. Khoshkam and G. Skandalis: \textit{Regular Representation of Groupoid $C^*$-algebras and Applications to Inverse Semigroups}, J. reine angew. Math. {\bf 546}, 47--72, (2002).

\bibitem{LMN} R. Lauter, B. Monthubert and Victor Nistor: \textit{Pseudodifferential Analysis on Continuous Family Groupoids}, Doc. Math. {\bf 5}, 625--655 (electronic), (2000).

\bibitem{Le} D. Lenz: \textit{Random Operators and Crossed Products}, Math. Phys. Anal. Geom. \textbf{2}, 197--220,  (1999).

\bibitem{LS1} D. Lenz and P. Stollmann: \textit{Algebras of Random Operators Associated to Delone Dynamical Systems}, Mathematical Physics, Analysis and Geometry, \textbf{6}, 269--290, (2003).

\bibitem{LS} D. Lenz and P. Stollmann: \textit{On the Decomposition Principle and a Persson Type Theorem for General Regular Dirichlet Forms}, J. Spectral Theory, \textbf{9}(3), 1089--1113, (2019).

\bibitem{Li} X.\;Li: \textit{Constructing Cartan Subalgebras in Classifiable Stably Finite $C^*$-Algebras}, ArXiV:1802.01190, (2017).

\bibitem{Mac} K. C. H. Mackenzie: \textit{General Theory of Lie Groupoids and Lie Algebroids}, vol. 213 of London Mathematical Society Lecture Note Series. Cambridge University Press, Cambridge, 2005.

\bibitem{Ma} M. M\u antoiu: \textit{Rieffel's Pseudodifferential Calculus and Spectral Analysis for Quantum Hamiltonians}, Ann. Inst. Fourier, {\bf 62}(4), 1551--1558,  (2012).

\bibitem{Ma2} M. M\u antoiu: \textit{Essential Spectrum and Fredholm Properties for Operators on Locally Compact Groups}, J. Oper. Theory, {\bf 77}(2) 481--501, (2017).

\bibitem{MN1} M. M\u antoiu and V. Nistor: \textit{Spectral Theory in a Twisted Groupoid Setting. Spectral Decompositions, Localization and Fredholmness}, M\"unster J. Math. \textbf{13}, 145--196, (2020).

\bibitem{MP} M. M\u antoiu and R. Purice: \textit{The Magnetic Weyl Calculus}, J. Math. Phys. \textbf{ 45}(4), 1394--1417, (2004).

\bibitem{MPR2} M. M\u antoiu, R. Purice and S. Richard: \textit{Spectral and Propagation Results for Magnetic Schr\"odinger Operators; a $C^*$-Algebraic Framework}, J. Funct. Anal. {\bf 250}, 42--67, (2007).

\bibitem{MR} M. M\u antoiu and M. Ruzhansky: \textit{Pseudo-Differential Operators, Wigner Transform and Weyl Systems on Type I Locally Compact Groups}, Doc. Math. \textbf{32}, 1539--1592, (2017).

\bibitem{MRW} P. Muhly, J. Renault and D. Williams: \textit{Equivalence and Isomorphism for Groupoid $C^*$-Algebras}, J. Operator Theory {\bf 17}, 3--22, (1987).

\bibitem{Ni} A. Nica: \textit{Wiener-Hopf Operators on the Positive Semigroup of a Heisenberg Group}, in Linear Operators in Function Spaces, \textbf{43}, Oper. Theory Adv. Appl., 263--278, Birkh\"auser, Basel, 1990.

\bibitem{NWX} V. Nistor,  A. Weinstein and P. Xu: \textit{Pseudodifferential Operators on Differential Groupoids}, Pacific J. Math., 1{\bf 89}(1), 117--152, (1999).

\bibitem{PR1} J. Packer and I. Raeburn:  \textit{Twisted Crossed Products of $\,C^*$-Algebras I}, \,Math. Proc. Cambridge Phyl. Soc.  \textbf{106}, 293--311, (1989).

\bibitem{Re} J. Renault: \textit{A Groupoid Approach to $C^*$-Algebras}, Lecture Notes in Mathematics, \textbf{793}, Springer, 1980.

\bibitem{Re1} J. Renault: \textit{Topological Amenability is a Borel Property}, Math. Scand. \textbf{117}(1),  5--30, (2015).

\bibitem{RS} J. Renault and S. Sunder: \emph{Groupoids Associated to Ore Semigroup Actions}, J. Operator Theory, \textbf{73}(2), 491--514, (2015).

\bibitem{Ri} M. A. Rieffel: \emph{Induced Representations of $C^*$-Algebras}, Adv. in Math, \textbf{13}, 176--257, (1974).

\bibitem{Roe} J. Roe: \textit{Lectures on Coarse Geometry}, University Lecture Series, \textbf{31}, American Mathematical Society, Providence, RI, 2003.

\bibitem{RT} M. Ruzhansky and V. Turunen: \textit{Pseudo-differential Operators and Symmetries}, Pseudo-Differential Operators: Theory and Applications \textbf{2}, Birkh\"auser Verlag, 2010.

\bibitem{STY} G. Skandalis, J. L. Tu and G. Yu: \textit{The Coarse Baum-Connes Conjecture and Groupoids}, Topology, \textbf{41}(4), 807--834, (2002).

\bibitem{SWi} J. Spakula and R. Willett: \textit{A Metric Approach to Limit Operators}, Trans. Amer. Math. Soc. \textbf{369}(1), 263--308, (2017).

\bibitem{SW} J. G. Stampfli and J. P. Williams: \textit{Growth Conditions and the Numerical Range in a Banach Algebra}, T\^ohoku Math. J. \textbf{20}, 417--424, (1968).

\bibitem{Va} S. Vassout: \textit{Unbounded Pseudodifferential Calculus on Lie Groupoids}, J. Funct. Anal., \textbf{236}(1), 161--200, 2006.

\bibitem{Wi} D. Williams: \emph{Crossed Products of $C^*$-Algebras}, Mathematical Surveys and Monographs, \textbf{134}, American Mathematical Society, 2007.


\end{thebibliography}
 \end{document}